\documentclass[psamsfonts]{amsart}
\usepackage{mathpazo}

\usepackage{overpic}
\usepackage{amsmath,amsthm,amssymb,amscd,amsfonts,amsbsy}
\usepackage{color}
\usepackage{hyperref,url}
\usepackage{float}
\usepackage{hyperref}
\usepackage{enumerate}
\usepackage{enumitem}   
\usepackage{mathrsfs}  
\usepackage{mathrsfs}  
\usepackage{amsthm}
\usepackage{amsmath}
\usepackage{amsfonts}
\usepackage{amssymb}
\usepackage{amsthm}
\usepackage{amsmath}
\usepackage{amsfonts}
\usepackage{amssymb}
\usepackage{mathtools}
\usepackage{mathrsfs}

\usepackage{amsmath, amssymb, graphics, setspace}
\usepackage[utf8]{inputenc}
\usepackage{listings,xcolor}

\newcommand{\mathsym}[1]{{}}
\newcommand{\unicode}[1]{{}}

\usepackage[normalem]{ulem}
\newcommand*{\defeq }{\mathrel{\vcenter{\baselineskip0.5ex \lineskiplimit0pt
                     \hbox{\scriptsize.}\hbox{\scriptsize.}}}%
                     =}
               
\newcommand{\R}{\ensuremath{\mathbb{R}}}

\newcommand{\f}{\varphi}
\newcommand{\al}{\alpha}
\newcommand{\la}{\lambda}

\newcommand{\e}{\varepsilon}

\newcommand{\sgn}{\mathrm{sign}}

\usepackage[toc,page]{appendix}
\usepackage{tikz-cd}

\newtheorem {theorem} {Theorem}
\newtheorem {definition} {Definition}

\newtheorem {proposition} [theorem]{Proposition}
\newtheorem {corollary}{Corollary}
\newtheorem {lemma}  [theorem]{Lemma}
\newtheorem {example} [theorem]{Example}
\newtheorem {remark}{Remark}

\newtheorem {mtheorem} {Theorem}

\def\R{\mathbb R}

\title[Lyapunov coefficients for monodromic tangential singularities]
{Lyapunov coefficients for monodromic tangential \\ singularities in Filippov vector fields}
\author[D. D. Novaes and L. A. Silva]
{Douglas D. Novaes and Leandro A. Silva}

\address{Departamento de Matem\'{a}tica, Instituto de Matem\'{a}tica, Estat\'{i}stica e Computa\c{c}\~{a}o Cient\'{i}fica, Universidade
Estadual de Campinas, \ Rua S\'{e}rgio Buarque de Holanda, 651, Cidade Universit\'{a}ria Zeferino Vaz, 13083-859, Campinas, SP,
Brazil}
\email{ddnovaes@unicamp.br}
\email{lasilva@ime.unicamp.br}

\allowdisplaybreaks
\begin{document}



\maketitle

\begin{abstract}
In planar analytic vector fields, a monodromic singularity can be distinguished between a focus or a center by means of the Lyapunov coefficients, which are given in terms of the power series coefficients of the first-return map defined around the singularity. In this paper, we are interested in an analogous problem for monodromic tangential singularities of piecewise analytic vector fields $Z=(Z^+ ,Z^-)$. First, we prove that the first-return map, defined in a neighborhood of a monodromic tangential singularity, is analytic, which allows the definition of the Lyapunov coefficients. Then, as a consequence of a general property for pair of involutions, we obtain that the index of the first non-vanishing Lyapunov coefficient is always even. In addition, a general recursive formula together with a Mathematica algorithm for computing the Lyapunov coefficients is obtained. We also provide results regarding limit cycles bifurcating from monodromic tangential singularities. Several examples are analyzed.
\end{abstract}

\section{Introduction and statements of the main results}

The center-focus problem and the cyclicity problem are classical problems in the qualitative theory of smooth planar vector fields, which goes back to the studies of Poincaré and Lyapunov (see, for instance, \cite{romanovski2009center}).

The center-focus problem consists in characterizing when a monodromic singularity of a planar vector field is either a center or a focus. A singularity is called a center if there exists a small neighborhood of this point such that all orbits are closed. A singularity is called a stable (resp. unstable) focus if there exists a small neighborhood of this point such that all orbits spiral towards to (resp. outwards from) the singularity. 
The center-focus problem can be studied by means of the first-return map defined in a section containing the monodromic singularity. Indeed, the monodromic singularity is a center if, and only if, the first-return map is the identity. If the vector field is analytic, then the first-return map is also analytic as well as the displacement function, which is given as the difference between the first-return map and the identity. The coefficients of the power series of the displacement function around the singularity provide the so-called {\it Lyapunov Coefficients}, $V_n$'s. Consequently,  the monodromic singularity is a center if, and only if, $V_n$ vanishes for every $n\in\mathbb{N}.$ This immediately  provides sufficient conditions in order for a monodromic singularity to be a focus.  For polynomials vector fields, Poincar\'{e} and Lyapunov reduced the problem of solving the infinite system of equations $V_n=0, n\in\mathbb{N},$ to an equivalent problem of finding a specific first integral for the vector field. 

On the other hand, the cyclicity problem consists in estimating the number of limit cycles that can bifurcate from a monodromic singularity, which can also be investigated by means of the Lyapunov coefficients. For more information on the center-focus and cyclicity problems, the interested readers are referred to the book \cite{romanovski2009center}.

Differential equations with discontinuities represent a very important class of dynamical systems due to their applications in many areas of applied science, see for instance the classical book of Andronov \cite{andronov1966theory} and, for more modern references, the books \cite{CRM,BBCK,Mike18}. Filippov, in his celebrated book \cite{Filippov88}, provided a rigorous mathematical formalization for the theory of non-smooth differential equations, and nowadays, such differential equations are called Filippov systems.
The center-focus and cyclicity problems have also been investigated for non-smooth planar vector fields. 
In \cite{pleshkan73}, Ple\u{s}kan and Sibirski\u{i} considered the center-focus problem for monodromic singularities of focus-focus type of piecewise analytic vector fields.
Filippov, in Chapter 4 of his book \cite{Filippov88}, computed several Lyapunov coefficients for a monodromic singularity of fold-fold type of a piecewise smooth vector field. In \cite{CGP95}, Coll et al. obtained the first seven Lyapunov constants for monodromic singularities of focus-focus type of discontinuous Li\'{e}nard differential equations. Then, in \cite{coll1999}, using an algebraic approach introduced by Cima et al. in \cite{cima97}, they derived  general expressions for the Lyapunov constants for monodromic singularities of focus-focus type of some families of discontinuous Li\'{e}nard differential equations. The same authors in \cite{gassulcoll} addressed the center-focus and cyclicity problems for monodromic singularities of focus-focus type, fold-fold type, and focus-fold type, computing explicitly the first three Lyapunov coefficients for these three types of monodromic singularities. In \cite{GasTor03}, Gasull and Torregrosa also addressed the center-focus and cyclicity problems for monodromic singularities of focus-focus type of several classes of piecewise smooth systems. The generic unfolding of a monodromic singularity of fold-fold type was considered in \cite{guardia11,kuzn03} (see also \cite{freire17} on this matter).  Recently, the problem of bifurcation of limit cycles from monodromic singularities in discontinuous systems by means of Lyapunov coefficients has been subject of investigation by several studies (see, for instance, \cite{braga21,GouTor20}, and the references therein). 

In this paper, we are interested in the center-focus and cyclicity problems for general {\it monodromic tangential singularities}, allowing them to be more degenerated than the fold-fold singularities already considered in the research literature.

\subsection{Filippov Vector Fields and Tangential Singularities} Consider the following planar piecewise analytic vector field:
\begin{equation}\label{eq:filippov}
(\dot{x},\dot{y}) = Z(x,y)=\begin{cases}
Z^+(x,y),&  h(x,y) > 0, \\
Z^-(x,y),& h(x,y) < 0,
\end{cases}
\quad (x,y)\in D,
\end{equation}
where $D$ is an open set of $\R^2,$ $Z^+,Z^-:D\rightarrow \R^2$ are analytic vector fields on $D,$ and $h:D\rightarrow\R$ is a smooth function having $0$ as a regular value. Notice that $\Sigma=h^{-1}(0)$ is the discontinuity manifold of \eqref{eq:filippov}. Usually, when the context is clear, the piecewise vector field \eqref{eq:filippov} can be concisely denoted by $Z=(Z^+,Z^-).$ 

Here, we shall assume the Filippov's convention \cite{Filippov88} for the trajectories of \eqref{eq:filippov}, which will be accordingly called {\it Filippov Vector Field}. In the Filippov context, the notion of singularities also comprehends the tangential points $\Sigma^t,$ which are constituted by the contact points between $Z^+$ and $Z^-$ with $\Sigma,$ that is, $$\Sigma^t=\{p\in \Sigma:\, Z^+h(p)\cdot Z^-h(p) = 0\},$$ where $Fh(p)=\langle\nabla h(p),F(p)\rangle$ denotes the Lie derivative of $h$ at $p$ in the direction of the vector field $F.$ 

Some tangential singularities give rise to local monodromic behavior. In what follows, we shall introduce the concept of  a $(2k^+,2k^-)$-monodromic tangential singularity for Filippov vector fields \eqref{eq:filippov}. Recall that $p$ is a {\it contact of multiplicity $k$} (or {\it order} $k-1$) between a smooth vector field $F$ and $\Sigma$ if $0$ is a root of multiplicity $k$ of $f(t)\defeq h\circ \f_{F}(t,p),$ where $t\mapsto \f_{F}(t,p)$ is the trajectory of $F$ starting at $p.$ Equivalently,
\begin{equation}\label{lieder}
F h(p) = F^2h(p) = \ldots = F^{k-1}h(p) =0,\text{ and } F^{k} h(p)\neq 0,
\end{equation}
where the higher Lie derivative $F^n h(p)$ is recursively defined  as $F^{n}h(p) =  F (F^ {n-1}h)(p),$ for $n>1.$ In addition, when considering Filippov vector fields \eqref{eq:filippov}, an even multiplicity contact, say $2k,$ is called {\it invisible} for $Z^+$ (resp. $Z^-$) when $(Z^{+})^{2k}h(p)<0$ (resp. $(Z^{-})^{2k}h(p)>0$). Otherwise, it is called {\it visible}. 

\begin{definition} \label{def:kpknmono}
A tangential singularity $p\in\Sigma^t$ of a Filippov vector field $Z$ \eqref{eq:filippov} is called a $(2k^+,2k^-)$-monodromic tangential singularity provided that $p$ is simultaneously an invisible $2k^{+}$-multiplicity contact of $Z^{+}$ with $\Sigma$ and an invisible $2k^{-}$-multiplicity contact of $Z^{-}$ with $\Sigma,$ and $Z$ has a first-return map defined on $\Sigma$ around $p.$
\end{definition}

\subsection{Main Results: Lyapunov Coefficients}

In \cite{gassulcoll}, the authors studied the Lyapunov coefficients for parabolic-parabolic points, which in light of Definition \ref{def:kpknmono} correspond to the $(2,2)$-monodromic tangential singularities. In this paper, our main goal consists in extending the previous results for $(2k^+,2k^-)$-monodromic tangential singularity. It is worth mentioning that in \cite{gassulcoll} the Lyapunov coefficients were obtained by means of generalized polar coordinates  (see \cite{broer1991structures}). Here, motivated by Teixeira's works \cite{teixeira77,teixeira81}, we propose a different way for obtaining it by considering auxiliary sections, which are transversal to both the flow and the discontinuity manifold. This method allows us to provide a general recursive formula for the Lyapunov coefficients. 

Suppose that the Filippov vector field \eqref{eq:filippov} has a $(2k^+,2k^-)$-monodromic tangential singularity at $p\in\Sigma.$ By taking local coordinates, we may assume, without loss of generality, that $p=(0,0)$ and $h(x,y)=y.$ Denoting $Z^{\pm}=(X^{\pm}(x,y), Y^{\pm}(x,y)),$ the Filippov vector field \eqref{eq:filippov} writes as
\begin{equation} \label{sistemainicial}
Z(x,y)= \begin{cases}
(X^+(x,y), Y^+(x,y)), & y > 0, \\
(X^-(x,y), Y^-(x,y)),& y < 0.
\end{cases}
\end{equation}
Working out the higher Lie derivatives \eqref{lieder} for $Z^{\pm},$ we get 
\[
(Z^{\pm})^nh(0,0)=X^{\pm}(0,0)^{n-1}\dfrac{\partial^{n-1}Y^{\pm}}{\partial x^{n-1}}(0,0),
\]
provided that $(Z^{\pm})^ih(0,0)=0$ for $i=1,\ldots,n-1.$ Thus, from Definition \eqref{def:kpknmono}, one can see that the origin is a $(2k^+,2k^-)$-monodromic tangential singularity for \eqref{sistemainicial} provided that the following three conditions are satisfied:

\bigskip

\noindent{\bf C1.} $X^{\pm}(0,0)\neq0,$ $Y^{\pm}(0,0)=0,$ $\dfrac{\partial^i Y^{\pm}}{\partial x^i}(0,0)=0$  for  $i=1,\ldots,2k^{\pm}-2,$ and  $\dfrac{\partial^{2k^{\pm}-1} Y^{\pm}}{\partial x^{2k^{\pm}-1} }(0,0)\neq0$;

\medskip

\noindent{\bf C2.} $X^+(0,0)\dfrac{\partial^{2k^{+}-1} Y^{+}}{\partial x^{2k^{+}-1} }(0,0)<0$ and $X^-(0,0)\dfrac{\partial^{2k^{-}-1} Y^{-}}{\partial x^{2k^{-}-1} }(0,0)>0;$

\bigskip

\noindent{\bf C3.} $X^+(0,0) X^-(0,0)<0.$

\bigskip

Condition {\bf C1} imposes that the origin is a contact of multiplicity $2k^+$  (resp. $2k^-$) between $Z^+$ (resp. $Z^-$) and $\Sigma.$ Condition {\bf C2} imposes that both contacts are invisible. Finally, condition {\bf C3} imposes that the orientation of the orbits of $Z^+$ and $Z^-$ around the origin agrees in such way that the trajectories of both vector fields can be concatenated  at $\Sigma$ in order to a first-return map be well defined around the origin. Denote
\begin{equation}\label{delta}
\delta=\textrm{sign}(X^+(0,0))=-\textrm{sign}(X^-(0,0)).
\end{equation}
Notice that, $Z$ is turning around the origin in the clockwise direction if $\delta>0,$ and in the anticlockwise direction if $\delta<0.$

The flows of $Z^{+}$ and $Z^{-}$ restricted, respectively, to $\Sigma^{+}=\{(x,y):\, y\geq 0\}$ and $\Sigma^{-}=\{(x,y):\, y\leq 0\}$ define half-return maps $\varphi^{+}$ and $\varphi^{-}$ on $\Sigma$ around $0,$ which are known to be \textbf{involutions} satisfying $\varphi^{+}(0)=\varphi^-(0)=0$ (see \cite[Section 4.2]{AndGomNov19}), that is, $\varphi^{+}\circ\varphi^{+}(x)=x$ and $\varphi^{-}\circ\varphi^{-}(x)=x$ whenever they are defined (see Figure \ref{involucoes}).

\begin{figure}[H]
	\centering 
	\begin{overpic}[scale=0.6]{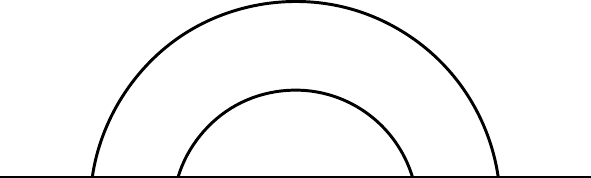}
	\put(13,-6){$x_1$}	
	\put(22,-6){$\varphi^+(x_2)$}	
	\put(48,-6){$0$}
	\put(78,-6){$\varphi^+(x_1)$}
	\put(68,-6){$x_2$}
	\put(103,-1){$\Sigma$}
	\end{overpic}
\vspace{0.5cm}
	\caption{Illustration of half-return map $\varphi^+.$}
	\label{involucoes}
\end{figure}
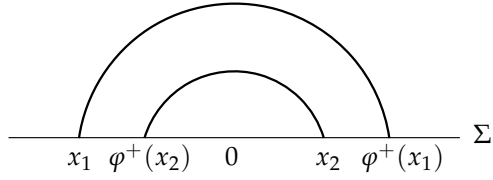

Our first main result states that such half-return maps are analytic provided that the vector fields $Z^{+}$ and $Z^{-}$ are analytic:
\begin{mtheorem}\label{teo:main}
Consider the Filippov vector field \eqref{sistemainicial} and  suppose that the vector field $Z^{+}$ (resp. $Z^{-}$)   is analytic  and has an invisible $2k^{+}$-multiplicity (resp. $2k^{-}$-multiplicity) contact at the origin with $\Sigma=\{(x,0):\,x\in\R\},$ for a positive integer $k^+$ (resp. $k^-$).  Then, the half-return map $\varphi^{+}$ (resp.  $\varphi^{-}$) is analytic around $x=0.$ 
\end{mtheorem}

 In \cite{gassulcoll}, the authors proved Theorem \ref{teo:main} assuming $k^{\pm}=1$ by means of {\it Generalized Polar Coordinates}. In Section  \ref{proof:ta}, we shall adapt their proof in order to obtain Theorem \ref{teo:main}. In the proof of Theorem \ref{teo:main}, it will be clear that if we impose $Z^{+}$ (resp. $Z^{-}$) to be $C^{r},$ $1\leq r\leq \infty,$ instead of analytic, then the half-return map $\varphi^{+}$ (resp.  $\varphi^{-}$) would be $C^r$ around $x=0.$ 

\bigskip

 Here, instead of working with the first-return map, we consider the {\it displacement function} 
$
\Delta(x)=\delta(\varphi^+(x)-\varphi^-(x))
$.
Assuming the analyticity of the vector fields $Z^{+}$ and $Z^{-},$ from Theorem \ref{teo:main} and taking into account that ${\varphi}^{+}$ and  ${\varphi}^{-}$ are involutions, we have
\begin{equation}\label{eq:series}
\begin{aligned}
{\varphi}^{\pm}(x) = -x + \alpha_2^{\pm}x^2 + \alpha_3^{\pm}x^3 +\cdots = -x + \sum_{n=2}^{\infty}\alpha_n^{\pm} x^n. \
\end{aligned}
\end{equation}
Thus,
\[
\Delta(x)= \sum_{n=2}^{\infty}V_n x^n,
\]
where $V_n=\delta(\al_n^+-\alpha_n^-),$ for $n\geq 2.$ Notice that \eqref{sistemainicial} has a center at the origin if, and only if, the displacement function is identically zero, equivalently, $V_n=0$ for every integer $n\geq 2.$ Hence, if there exists $n \in \mathbb{N}$ such as $V_n\neq0,$ then \eqref{sistemainicial} has a focus at the origin. In addition, if $n_0$ is the first index such that  $V_{n_0} \neq 0,$ then the origin is asymptotically stable (resp. unstable) provided that $V_{n_0}<0$ (resp. $V_{n_0}>0$). Accordingly, it is natural to define $n$th-Lyapunov coefficient of a $(2k^+,2k^-)$-monodromic tangential singularity by $V_n$ as above.

\smallskip

For non-degenerated monodromic singularities of planar smooth vector fields, the index of the first non-vanishing Lyapunov coefficient is always odd (see \cite{romanovski2009center}). In \cite[Theorem B(c)]{gassulcoll}, for $(2,2)$-monodromic tangential singularity, it was showed that $V_3=0$ provided that $V1=V_2=0$. Here, as a consequence of an involutive property of the half-return maps (see Proposition \ref{prop1} of Section \ref{proof:tb}), our second main result  generalizes  \cite[Theorem B(c)]{gassulcoll} and establishes that the index of the first non-vanishing Lyapunov coefficient of a $(2k^+,2k^-)$-monodromic tangential singularity is always even:

\begin{mtheorem}\label{teo:oddindex}
Consider the Filippov vector field $Z$ given by \eqref{sistemainicial} and  suppose that the vector fields $Z^{+}$ and $Z^{-}$ are analytic. Assume that $Z$ has a $(2k^+,2k^-)$-monodromic tangential singularity at the origin, for positive integers $k^+$ and $k^-.$ If $V_n=0$ for every $n=2,\ldots,2\ell,$ then $V_{2\ell+1}=0.$
\end{mtheorem}

Theorem \ref{teo:oddindex} is proven in Section \ref{proof:tb}.

\bigskip

Our third main result provides a recursive formula for computing the coefficients  $\alpha^{+}_n$ and $\alpha^{-}_n$ of the series \eqref{eq:series} of the half-return maps $\varphi^{+}$ and $\varphi^{-}$ and, consequently, the Lyapunov coefficients $V_n$'s. 

Consider the value
\begin{equation} \label{value:a}
a^{\pm}=\dfrac{1}{(2k^{\pm}-1)!|X^{\pm}(0,0)|}\dfrac{\partial^{2k^{\pm}-1}Y^{\pm}}{\partial x^{2k^{\pm}-1}}(0,0).\end{equation}
Notice that, from condition {\bf C3}, $\delta a^{\pm}<0.$ Also, define the functions
\begin{equation}\label{auxfunc}
\begin{array}{l}
f^{\pm}(x)=\dfrac{\pm \delta Y^{\pm}(x,0)-a^{\pm}x^{2k^{\pm}-1} X^{\pm}(x,0)}{x^{2k^{\pm}} X^{\pm}(x,0)}\quad\text{and}\vspace{0.2cm}\\
g^{\pm}(x,y)=\dfrac{\pm X^{\pm}(x,0)Y^{\pm}(x,y) \mp X^{\pm}(x,y)Y^{\pm}(x,0)}{y \delta X^{\pm}(x,y)X^{\pm}(x,0)}.
\end{array}
\end{equation}

\begin{mtheorem} \label{teo:coef}
Consider the Filippov vector field $Z$ given by \eqref{sistemainicial} and  suppose that the vector fields $Z^{+}$ and $Z^{-}$ are analytic. Assume that $Z$ has a $(2k^+,2k^-)$-monodromic tangential singularity at the origin, for positive integers $k^+$ and $k^-.$ Then, the functions $f^{\pm}$ and $g^{\pm}$ are analytic in a neighborhood of $x=y=0.$ Also, consider the sequence of functions $y_i^{+}$ and $y_i^{-}$ defined, in a neighborhood of $x=0,$ recursively by
\begin{equation}\label{yi}
\begin{aligned}
y_1^{\pm} (x)=& a^{\pm} x^{2k^{\pm}-1}+x^{2k^{\pm}}f^{\pm}(x),&\\
y_i^{\pm} (x) = & (\pm \delta)^{i-1}\left( a^{\pm} \dfrac{(2k^{\pm}-1)!}{(2k^{\pm}-i)!}x^{2k^{\pm}-i}  +\sum_{l=0}^{i-1} {{i-1}\choose{l}}  \dfrac{(2k^{\pm})!}{(2k^{\pm}-l)!}x^{2k^{\pm}-l} {f^{\pm}}^{(i-1-l)} (x)\right) & \\  & +\sum_{l=1}^{i-1} \sum_{j=1}^{l}j {i-1\choose l} (\pm \delta)^{i-l-1}B_{l,j}(y^{\pm}_1(x),\dots , y^{\pm}_{l-j+1}(x))  \dfrac{\partial^{j+i-l-2}g^{\pm}}{\partial x^{i-l-1} \partial y^{j-1}}(x,0), \text{ if } 2\leq i\leq 2k^{\pm}, \\
y_i^{\pm}(x) = & (\pm\delta)^{i-1} \Bigg(  {{i-1}\choose{2k^{\pm}}}(2k^{\pm})!{f^{\pm}}^{i-1-2k^{\pm}}(x)   +\sum_{l=0}^{2k^{\pm}-1} {{i-1}\choose{l}}  \dfrac{(2k^{\pm})!}{(2k^{\pm}-l)!}x^{2k^{\pm}-l}{f^{\pm}}^{(i-l-1)}(x)\Bigg) &  \\  &+ \sum_{l=1}^{i-1}\sum_{j=1}^{l}j {{i-1}\choose{l}} (\pm \delta)^{i-l-1}B_{l,j}(y_1^{\pm}(x),\dots , y^{\pm}_{l-j+1}(x)) \dfrac{\partial^{j+i-l-2}g^{\pm}}{\partial x^{i-l-1} \partial y^{j-1}}(x,0),\,\,  \text{ if }\, i>2k^{\pm}. 
\end{aligned}
\end{equation}
Then, the coefficients  $\alpha^{\pm}_n$ of the series \eqref{eq:series} of the half-return maps $\varphi^{\pm}$ are given recursively by
\begin{equation}\label{alpha}
\begin{cases}
\alpha^{\pm}_1= -1,\vspace{0.1cm}\\
\alpha^{\pm}_n=\dfrac{p^{\pm}_{n,k^{\pm}}(\alpha^{\pm} _1, \alpha^{\pm} _2, \cdots \alpha^{\pm} _{n-1}) - \mu^{\pm}_{n + 2k^{\pm} -1}}{2k^{\pm} \mu^{\pm}_{2k^{\pm}}}, 
\end{cases}
\end{equation}
where 
\[
p^{\pm}_{n,k^{\pm}}\big(\alpha_1,\ldots,\alpha_{n-1}\big)= \mu^{\pm}_{2k^{\pm}} \hat B_{n+2k^{\pm}-1,2k}\big(\alpha_1,\ldots,\alpha_{n -1},0\big) +  \sum_{i=2k^{\pm}+1}^{n+2k^{\pm}-1} \mu^{\pm}_i \hat B_{n+2k^{\pm}-1,i}\big(\alpha_1,\ldots,\alpha_{n+2k^{\pm}-i}\big),
\]
and
\begin{equation}\label{mui}
\mu_i^{\pm}= \dfrac{1}{i!}\sum_{j=1}^{i} (\mp \delta)^j {{i}\choose{j}} (y_j^\pm)^{(i-j)}(0).
\end{equation}
\end{mtheorem}

Theorem \ref{teo:coef} is proven in Section \ref{proof:coef}.

In the recurrences above, we are using the concept of {\it partial Bell polynomials} and  {\it ordinary Bell polynomials} (see, for instance \cite{comtet}), which for positive integers $p$ and $q$ are defined, respectively, as follows
\begin{equation}\label{bell}
\begin{array}{l}
\displaystyle B_{p,q}(x_1,\ldots,x_{p-q+1})=\sum\dfrac{p!}{b_1!\,b_2!\cdots b_{p-q+1}!}\prod_{j=1}^{p-q+1}\left(\dfrac{x_j}{j!}\right)^{b_j} \text{ and }\vspace{0.2cm}\\
\displaystyle \hat B_{p,q}(x_1,\ldots,x_{p-q+1})=\sum\dfrac{p!}{b_1!\,b_2!\cdots b_{p-q+1}!}\prod_{j=1}^{p-q+1}x_j^{b_j}.
\end{array}
\end{equation}
The sums above are taken  
over all the $(p-q+1)$-tuple of nonnegative integers $(b_1,b_2,\cdots,b_{p-q+1})$ satisfying $b_1+2b_2+\cdots+(p-q+1)b_{p-q+1}=p,$ and
$b_1+b_2+\cdots+b_{p-q+1}=q.$ Notice that
\[
 \hat B_{p,q}(x_1,\ldots,x_{p-q+1})=\dfrac{q!}{p!}B_{p,q}(1!x_1,\ldots,(p-q+1)!x_{p-q+1}).
\]
It is worth mentioning that partial Bell polynomials are implemented in algebraic manipulators as Mathematica and Maple. Based on Theorem \ref{teo:coef}, we provide Appendix A containing an implemented Mathematica algorithm for computing the Lyapunov coefficients.  

\bigskip

In the following corollary, applying Theorem \ref{teo:coef}, we compute $\alpha^{\pm}_n$, for $n=1,2,3,4,$  for a general $(2k^+,2k^-)$-monodromic tangential singularity. Recall that the $i$th-Lyapunov coefficient is given by $V_n=\delta(\al_n^+-\alpha_n^-).$

\begin{corollary}\label{cor} Assume that the Filippov vector field \eqref{sistemainicial} has a $(2k^+,2k^-)$-monodromic tangential singularity at the origin, for positive integers $k^+$ and $k^-,$ and denote 
	
	\[
	f^{\pm}(x)=\sum_{i=0}^{\infty} f_i^{\pm}\, x^i\quad \text{and} \quad g^{\pm}(x,y)=\sum_{i=0}^{\infty} \sum_{j=0}^{\infty}  g_{i,j}^{\pm}\, x^i y^j.
	\]
	
	Then,  the first four coefficients $\alpha^{\pm}_n$'s of the series \eqref{eq:series} of the half-return maps $\f^{\pm}$  are given by 
	\[
	\begin{aligned}
	\alpha_1^{\pm}=&  -1,\quad 
	\alpha_2^{\pm}=  \dfrac{-2f_0 \pm 2 \delta a g_{0,0}}{2a k + a},\quad \alpha_3^{\pm}= -(\alpha_2^{\pm})^2, \vspace{0.2cm}\\
	\alpha_4^{\pm}=& \dfrac{4({k} (2 {k}+3)+7) (-f_0 \pm \delta {a}   {g_{0,0}})^3}{3{a}^3(2 {k}+1)^3} \mp  \dfrac{ 12 \delta a(f_0\mp \delta {a}   {g_{0,0}}) \left({a} \left({g_{1,0}} \mp \delta  {g_{0,0}}^2\right)+2 f_0 {g_{0,0}} \mp 2 \delta  {f_1}\right)}{3{a}^3(8{k}+4)} \vspace{0.2cm}\\  &\pm  \dfrac{ 4\delta {a}^2 \left({a} \left({2g_{2,0}}+{g_{0,0}}^3\right)+6 {g_{0,0}} {f_1}+3 f_0 {g_{1,0}}\mp 3 \delta  \left({a} {g_{1,0}} {g_{0,0}}+f_0 {g_{0,0}}^2+{2f_2}\right) \right)}{3 {a}^3 (8 {k}+12)}+\xi_k.
	\end{aligned}
	\]
	where $\xi_1=- \dfrac{4 {a} {g_{0,1}}}{15}$ and $\xi_k=0$ for $k>1.$ For the sake of simplicity, in the above expressions we are dropping the sign $\pm$ from $a^{\pm}, k^{\pm}, f_i^{\pm},$ and $g_{i,j}^{\pm}.$
\end{corollary}

Since, in Corollary \ref{cor}, $k^{\pm}$ are arbitrary, its proof involves some very cumbersome computations when applying the formulae \eqref{yi}-\eqref{mui}. Those computations and the proof of Corollary \ref{cor} are included in Appendix B. Notice that, if $k^{\pm}$ are known, then the formulae \eqref{yi}-\eqref{mui} can be easily algorithmically implemented (see Appendix A).

\subsection{Application: Bifurcation of Limit Cycles}

It is well known that Lyapunov coefficients can be used to study the appearance of small amplitude limit cycles in smooth and non-smooth vector fields around weak focuses (see, for instance, \cite{romanovski2009center} for smooth vector fields and \cite{gassulcoll,GasTor03,GouTor20} for non-smooth vector fields). In this section, we apply the classical ideas to study the appearance of limit cycles around monodromic tangential singularities. We start by providing a Hopf-like bifurcation theorem for monodromic tangential singularities.

\begin{mtheorem}\label{thm:hopf} Let $k^+$ and $k^-$ be positive integers  and let $Z_\lambda$ be a 1-parameter family of Filippov vector fields \eqref{sistemainicial} having a $(2k^+,2k^-)$-monodromic tangential singularity at the origin for every $\lambda$ in an interval $I.$ Let $V_2(\lambda)$ and $V_4(\lambda)$ be, respectively, the second and the forth Lyapunov coefficients. Assume that, for some $\lambda_0\in I,$ $V_2(\lambda_0)=0,$ $d:=V_2'(\lambda_0)\neq0,$ and $\ell:=V_4(\lambda_0)\neq0.$ Then, there exists a neighborhood $J\subset I$ of $\lambda_0$ such that, for every $\lambda\in J$ satisfying  $d\ell(\lambda-\lambda_0)<0,$ the Filippov vector field $Z_\lambda$ admits a hyperbolic limit cycle in a $\sqrt{|\la-\la_0|}$-neighborhood of  the origin. In addition, such a limit cycle is asymptotically stable (resp. unstable) provided that $\ell<0$ (resp. $\ell>0$).  
\end{mtheorem}

Theorem \ref{thm:hopf} is proven in Section \ref{sec:limc}.

\begin{example}
Let $k^+$ and $k^-$ be positive integers, $\la\in\R,$   and consider the following 1-parameter family of Filippov vector fields:

\begin{equation} \label{ex1}
Z_{\lambda}(x,y)= \begin{cases}
\Big(1, -x^{2k^+-1}(\lambda\,x+1) \Big),&  y > 0, \\
\Big(-1,x^{2k^--1}(x-1) \Big),&  y < 0.
\end{cases}
\end{equation}

Notice that the origin is a $(2k^+,2k^-)$-monodromic tangential singularity for every $\lambda\in\R.$
From Corollary \ref{cor}, we compute
\[
V_2(\lambda)=-\dfrac{2}{1+2k^+}\lambda+\dfrac{2}{1+2k^-} \,\text{ and }\,
V_4(\lambda)=-\dfrac{
	4 (7 + k^+ (3 + 2 k^+)) \lambda^3}{3 (1 + 2 k^+)^3}\lambda^3-\dfrac{4 (7 + k^- (3 + 2 k^-))}{3 (1 + 2 k^-)^3}.  
\]
Thus, for $\lambda_0=\dfrac{1+2k^+}{1+2K^-},$ we have
\[
V_2(\lambda_0)=0,\,\, d=V_2'(\lambda_0)=-\dfrac{2}{1+2k^+}<0,\, \text{ and } \, \ell=V_4(\lambda_0)=-\dfrac{8 (7 + k^+ (3 + 2 k^-))}{3 (1 + 2 k^-)^3}< 0.
\]
Since $\sgn(\ell)=-1,$ Theorem \ref{thm:hopf} implies that, the Filippov vector field \eqref{ex1} admits an asymptotically stable hyperbolic limit cycle for every $\lambda$ sufficiently close to $\lambda_0.$ Such a limit cycle converges to the origin as $\lambda$ goes to $\lambda_0.$
\end{example}

\begin{example}
Let $k$ be a positive integer, $\la\in\R,$ and consider the following 1-parameter family of Filippov vector fields:

\begin{equation} \label{ex2}
Z_{\lambda}(x,y)= \begin{cases}
\Big(1, x^{2k-1}(\lambda\,x-1)+y \Big),&  y > 0, \\
\Big(-1,x^{2k-1}(x-1) \Big),& y < 0,
\end{cases}
\end{equation}

Notice that the origin is a $(2k,2k)$-monodromic tangential singularity for every $\lambda\in\R.$

From Corollary \ref{cor}, we compute
\[
\begin{aligned}
V_2(\lambda)=&\dfrac{2\lambda}{1+2k}\quad \text{ and }\vspace{0.2cm}\\
V_4(\lambda)=&\dfrac{1}{3(3+2k)(1+2k)^3}\Big(-4 (2 + k) (1 + 2 k)^2 + 12 (19 + 14 k) \lambda + 
6 (3 + 2 k) (13 + 2 k) \lambda^2 \\ &+ 
4 (3 + 2 k) (7 + k (3 + 2 k)) \lambda^3 \Big).
\end{aligned}
\]
Thus, for $\lambda_0=0,$ we have
\[
V_2(\lambda_0)=0,\,\, d=V_2'(\lambda_0)=\dfrac{2}{1+2k}>0,\, \text{ and } \, \ell=V_4(\lambda_0)=\dfrac{-4(2+k)}{9+12k(2+k)}< 0.
\]
Theorem \ref{thm:hopf} implies that the Filippov vector field \eqref{ex2} admits an asymptotically stable hyperbolic limit cycle for every $\lambda>\lambda_0$ sufficiently close to $\lambda_0.$ Such a limit cycle converges to the origin as $\lambda$ goes to $\lambda_0.$

\end{example}

Theorem \ref{thm:hopf}  can be generalized as follows:

\begin{mtheorem}\label{thm:genhopf} Let $k^+$ and $k^-$ be positive integers  and let $Z_\Lambda$ be an $n$-parameter family of Filippov vector fields \eqref{sistemainicial} having a $(2k^+,2k^-)$-monodromic tangential singularity at the origin for every $\Lambda$ in an open set $U\subset\R^n.$ Let $V_{2i}(\Lambda)$ be the $2i-$th Lyapunov coefficient, for $i=1,2\dots,n+1,$ and denote $\mathcal V_n=(V_2,V_4,\ldots,V_{2n}):U\rightarrow \R^n.$ Assume that, for some $\Lambda_0\in U,$ $\mathcal V_n(\Lambda_0)=0,$ $\det (D\mathcal V_n(\Lambda_0))\neq0,$ and $V_{2n+2}(\Lambda_0)\neq0.$ Then, there exists an open set $W\subset U$ such that $Z_\Lambda$ has $n$ hyperbolic limit cycles for every $\Lambda\in W.$ In addition, all the limit cycles converge to the origin as $\Lambda$ goes to $\Lambda_0.$
\end{mtheorem}

Theorem \ref{thm:genhopf} is proven in Section \ref{sec:limc}.

\begin{example}
Let $\Lambda=(\la_1,\la_2,\ldots,\la_5)\in\R^5$ and consider the following 5-parameter family of Filippov vector fields:

\begin{equation} \label{ex3}
Z_{\Lambda}(x,y)= \begin{cases}
\Big(1, -x+\la_1 x^2+\la_2 x y+\la_3 y^2 \Big),&  y > 0, \\
\Big(-1, -x+x^2+\la_4 x y +\la_5 y^2\Big),& y < 0.
\end{cases}
\end{equation}

Notice that the origin is a $(2,2)$-monodromic tangential singularity for every $\Lambda\in\R^5.$
From Theorem \ref{teo:coef}, we compute $\mathcal{V}_{5}=(V_2(\Lambda),V_4(\Lambda),\ldots,V_{10}(\Lambda))$ and $V_{12}(\Lambda).$
Thus, for 
\[
\Lambda_0=\Big(1,\dfrac{5(-1+\sqrt{109})}{2},-\dfrac{5(-7+\sqrt{109})}{4},\dfrac{5(1+\sqrt{109})}{2},\dfrac{5(7+\sqrt{109})}{4}\Big),
\] 
we have
\[
\mathcal V_{5}(\Lambda_0)=0,\,\, \det (D\mathcal V_{5}( \Lambda_0))=\dfrac{1520768}{74263959},\, \text{ and } V_{12}(\Lambda_0)=\dfrac{20030\sqrt{109}}{9009}.
\]
Therefore, from Theorem \ref{thm:genhopf}, there exists an open set $W\subset\R^5$ such that the Filippov vector field \eqref{ex3}, $Z_\Lambda,$ has $5$ hyperbolic limit cycles for every $\Lambda\in W.$ In addition, all the limit cycles converge to the origin as $\Lambda$ goes to $\Lambda_0.$

\end{example}

\begin{remark}
In Theorems D and E and in Examples 1, 2, and 3, the obtained hyperbolic limit cycles coexist with the monodromic tangential singularity. In other words, the limit cycles are obtained without destroying the singularity. Therefore, by performing the following small perturbations 
\[
Z_{\la,\e}(x,y)=\begin{cases}
Z^+_{\la}(x-\e,y),&y>0,\\
Z^-_{\la}(x,y),&y<0,
\end{cases} \quad\text{and}\quad Z_{\Lambda,\e}(x,y)=\begin{cases}
Z^+_{\Lambda}(x-\e,y),&y>0,\\
Z^-_{\Lambda}(x,y),&y<0,
\end{cases}
\]
one can see that $Z_{\la,\e}$ and $Z_{\Lambda,\e}$ undergo a pseudo-Hopf bifurcation at $\e=0$, which creates a sliding segment and an additional hyperbolic limit cycle (see Figure \ref{fig:perturbation}), increasing by $1$ the number of limit cycles in Theorems D and E and in Examples 1, 2, and 3.  The pseudo-Hopf bifurcation was reported by Filippov in his book \cite{Filippov88} (see item b of page 241) and has been further investigated in \cite{castillo17}. This bifurcation is a useful tool to increase the number of limit cycles when dealing with the cyclicity problem (see, for instance, \cite{cruz19,GouTor20}).
\end{remark}

\begin{figure}[H]
	\centering
	\begin{overpic}[scale=0.6]{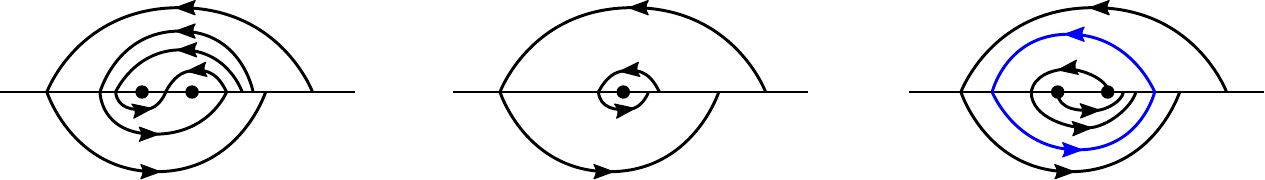}
			\put(10,-3){$\varepsilon<0$}
			\put(47,-3){$\varepsilon=0$}
			\put(84,-3){$\varepsilon>0$}
	\end{overpic}
	\vspace{0.5cm}
	\caption{Pseudo-Hopf bifurcation, which creates a sliding segment and a hyperbolic limit cycle.}
	\label{fig:perturbation}
\end{figure}

\subsection{Structure of the paper}
First, in Section \ref{sec:cf}, we recall a canonical expression, introduced in \cite{novaes2020smoothing}, for Filippov vector fields \eqref{sistemainicial} around $(2k^{+}, 2k^{-})$-monodromic tangential singularities. Such a canonical form will be of major importance for proving our main results. Then, in Section \ref{sec:gpc}, we introduce the concept of generalized polar coordinates (see \cite{broer1991structures}), which allow us to obtain, in Section \ref{proof:ta}, the proof of Theorem \ref{teo:main} about the regularity of the half-return maps defined around $(2k^{+}, 2k^{-})$-monodromic tangential singularities. In Section \ref{sec:invo}, we discuss a general property for pair of involutions from which we straightforwardly conclude that the index of the first non-vanishing Lyapunov coefficient of a $(2k^+,2k^-)$-monodromic tangential singularity is always even. This leads to the proof of Theorem \ref{teo:oddindex} in Section \ref{proof:tb}. Section \ref{sec:LC} is devoted to obtaining the recursive formula for the Lyapunov coefficients stated in Theorem \ref{teo:coef}, which is then proven in Section \ref{proof:coef}. In Section \ref{sec:limc}, we discuss the appearance of small amplitude limit cycles around  $(2k^{+}, 2k^{-})$-monodromic tangential singularities.  Theorems \ref{thm:hopf} and \ref{thm:genhopf} are proven in Section \ref{sec:limc}. Finally, two Appendixes are provided: Appendix A contains the implemented algorithms, based on Theorem \ref{teo:coef}, for computing the Lyapunov coefficients, and Appendix B is devoted to the proof of Corollary \ref{cor}.

\section{Canonical Form}\label{sec:cf}

In this section, we provide a simpler expression for Filippov vector fields around a $(2k^{+}, 2k^{-})$-monodromic tangential singularity. This canonical expression has been introduced in \cite{AndGomNov19} and will be important for proving our main results. In what follows, for the sake of completeness, we briefly explain how to obtain it.

Assuming that the Filippov vector field \eqref{sistemainicial} has a $(2k^{+}, 2k^{-})$-monodromic tangential singularity at the origin (see conditions {\bf C1}, {\bf C2}, and {\bf C3}), we have that $X^{\pm}(0,0)\neq0.$
Therefore, there exists a small neighborhood $U$ of the origin such that $X^{\pm}(x,y) \neq0$ for all $(x,y) \in U.$  Taking into account that $|X^{\pm}(x,y)| = \pm \delta X^{\pm}(x,y)$ for every $(x,y)\in U,$ a time rescaling can be performed in order to transform the Filippov vector field \eqref{sistemainicial} restricted to $U$ into \begin{equation*} 
(\dot{x},\dot{y}) = \tilde{Z}(x,y)=\begin{cases}
(\delta, \eta^+(x,y)), &y > 0, \\
(-\delta, \eta^-(x,y)), &y < 0,
\end{cases}
\end{equation*}
 where 
\begin{equation*} \label{eq:eta}
\eta^+(x,y) = \delta\dfrac{Y^+(x,y)}{X^+(x,y)}\,\text{ and }\, \eta^-(x,y) = -\delta\dfrac{Y^-(x,y)}{X^-(x,y)}. 
\end{equation*}

In addition, we can show that 
\begin{equation}\label{info1}
{(\tilde{Z}^{\pm})}^ih(0,0) = 0 \,\text{ if, and only if, }\, (Z^{\pm})^ih(0,0) = 0,\,\text{ for all }\, i = 1, 2, \dots, 2k^{\pm},
\end{equation}
and
\begin{equation}\label{info2}
\tilde{Z}^{\pm}h(x,0)  = \eta^{\pm}(x,0)\, \text{ and }(\tilde{Z}^{\pm})^ih(x,0) = \dfrac{\partial^{i-1}}{\partial x^{i-1}}\eta^{\pm}(0,0), \,\text{ for all }\, i = 1, \dots, 2k^{\pm}.
\end{equation}
Since $(Z^{\pm})^ih(0,0) = 0$ for $i = 1, 2, \dots, 2k^{\pm}-1$ and $(Z^{\pm})^{2k^{\pm}}h(0,0) \neq 0,$ by combining \eqref{info1} and \eqref{info2}, we can expand $\eta^{\pm}(x,0)$ around $x=0$ as follows:
\begin{equation*} 
\eta^{\pm}(x,0)= \sum_{i=0}^{2k^{\pm}-1}\dfrac{1}{i!}\dfrac{\partial^i\eta^{\pm}}{\partial x^i}(0,0)x^i+x^{2k\pm}f^{\pm}(x)= a^{\pm}x^{2k^{\pm}-1} + x^{2k\pm}f^{\pm}(x).
\end{equation*} 
 Consequently, the function $\eta^{\pm}(x,y)$ writes
\begin{equation*}
\eta^{\pm}(x,y) = a^{\pm}x^{2k^{\pm}-1} + x^{2k^{\pm}}f^{\pm}(x) + yg^{\pm}(x,y).
\end{equation*} 
Notice that 
\[
\begin{aligned}
a^{\pm}=&\dfrac{1}{(2k^{\pm}-1)!}\dfrac{\partial^{2k^{\pm}-1}\eta^{\pm}}{\partial x^{2k^{\pm}-1}}(0,0)
=\dfrac{\pm\delta}{(2k^{\pm}-1)!}\dfrac{\partial^{2k^{\pm}-1}}{\partial x^{2k^{\pm}-1}}\left(\dfrac{Y^{\pm}(x,0)}{ X^{\pm}(x,0)}\right)\Bigg |_{x=0}\\
=&\dfrac{1}{(2k^{\pm}-1)!|X^{\pm}(0,0)|}\dfrac{\partial^{2k^{\pm}-1}Y^{\pm}}{\partial x^{2k^{\pm}-1}}(0,0),
\end{aligned}
\]
which coincides with the value defined in \eqref{value:a}, and the functions $f^{\pm}(x)$ and $g^{\pm}(x,y)$ coincide with the ones defined in  \eqref{auxfunc}.

Accordingly, the Filippov vector field \eqref{sistemainicial} on $U$ is equivalent to 
\begin{equation} \label{sistemacanonico}
(\dot{x},\dot{y}) = \begin{cases}
(\delta, a^+x^{2k^+-1} + x^{2k^+}f^+(x) + yg^+(x,y)), &  y > 0, \\
(-\delta,a^-x^{2k^--1} + x^{2k^-}f^-(x) + yg^-(x,y)), &y < 0.
\end{cases}
\end{equation}

\section{Regularity of the half-return map}\label{sec:gpc}

This section is devoted to the proof of Theorem \ref{teo:main}. The main tool we shall employ to obtain the analyticity of the half-return maps $\f^{\pm}$ is the {\it generalized polar coordinate transformation} (see \cite{broer1991structures}). This coordinate transformation was introduced by Lyapunov in \cite{liapunov66} and was originally conceived for studying degenerate singularities of vector fields.  Afterwards, this tool has shown to be very useful in the study of  degenerate singularities of smooth planar vector fields (see, for instance, \cite{brunella90,CGP97,dumortier90}).

In \cite{gassulcoll}, Coll et al. employed the generalized polar coordinate transformation to study several types of monodromic singularities of discontinuous piecewise analytic systems, one of them, the fold-fold type, coincides  with our $(2,2)$-monodromic tangential singularity, according to Definition \ref{def:kpknmono}. For this case, they showed that, in the transformed space $\mathbb{S}^1\times \R^+,$ the monodromic singularity blows-up into $\{0\}\times\mathbb{S}^1,$ which does not have singularities in the closure of the semi-plane of interest (see Figure \ref{fig:analiticity}). As we shall see, this implies the analyticity of half-return maps around $(2,2)$-monodromic tangential singularities.

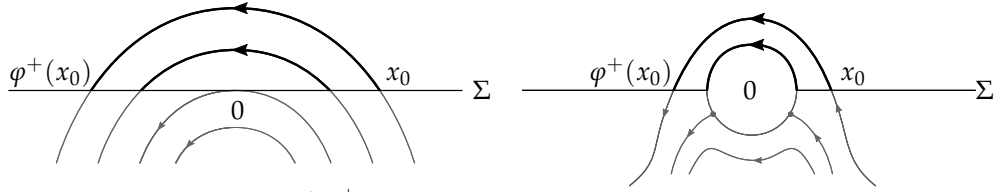
\begin{figure}[H]
	\centering
	\begin{overpic}[scale=0.6]{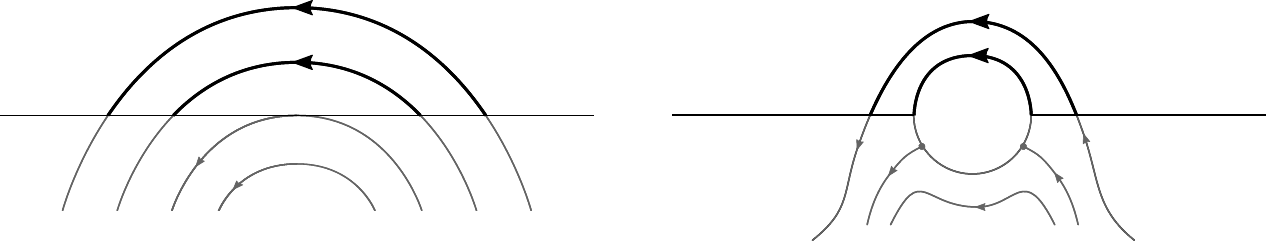}
		\put(48,9){$\Sigma$}
		\put(100,9){$\Sigma$}
		\put(23,7){$0$}
		\put(76,9){$0$}		
		\put(0,11){$\varphi^+(x_0)$}
		\put(60,11){$\varphi^+(x_0)$}
		\put(39,11){$x_0$}
		\put(86,11){$x_0$}
	\end{overpic}
	\caption{Blow-up of $Z^+$ at the monodromic tangential singularity. In the transformed space $\mathbb{S}^1\times \R^+,$ the monodromic singularity blows-up into $\{0\}\times\mathbb{S}^1,$ which does not have singularities in the closure of the semi-plane of interest.}
	\label{fig:analiticity}
\end{figure}

The same result for $(2,2)$-monodromic tangential singularities can be obtained by using an analytic version of Vishik's normal form (see \cite{castro}), however, this does not work for more degenerated tangential singularities. Here, we shall follow the ideas from  \cite{gassulcoll} to obtain the analyticity of half-return maps around $(2k^+,2k^-)$-monodromic tangential singularities.

First, we recall the definition of generalized polar coordinates. For positive real numbers $p$ and $q,$ the $(R, \theta, p, q)$-generalized polar coordinates are given by $\big(x,y\big)=\big(R^p\text{Cs}(\theta),R^q\text{Sn}(\theta)\big),$ where $R>0,$ $\theta\in\mathbb{S}^1,$ and the function $\big(\text{Sn}(\theta),\text{Cs}(\theta)\big)$ is the solution of the following Cauchy problem:
\begin{equation*}
\begin{cases}
\dot{\text{Cs}}=-{\text{Sn}}^{2p-1}, \\
\dot{\text{Sn}}={\text{Cs}}^{2q-1},\\
\end{cases}
\text{Cs}(0)=\sqrt[2q]{\dfrac{1}{p}},\quad \text{Sn}(0)=0.
\end{equation*}
In \cite{CGP97}, it is proven that the functions $\text{Cs}(\theta)$ and $\text{Sn}(\theta)$ are analytic, $T$-periodic with
\[
T=2p^{\frac{-1}{2q}}q^{\frac{-1}{2p}}\int_0^1(1-s)^{\frac{1-2p}{2p}}s^{\frac{1-2q}{2q}}ds>0,
\]
and satisfy the following properties:
\begin{itemize}
\item $p (\text{Cs}(\theta))^{2q}+q(\text{Sn}(\theta))^{2p}=1;$

\smallskip

\item $\text{Cs}$ is an even function and $\text{Sn}$ is an odd function;

\smallskip

\item $\text{Cs}(\frac{T}{2}-\theta)=-\text{Cs}(\theta)$ and  $\text{Sn}(\frac{T}{2}-\theta)=\text{Sn}(\theta).$
\end{itemize}
Taking the above properties into account, one can see that 
\begin{equation}\label{cssn}
\begin{array}{l}
\text{Cs}(-\frac{T}{4})=\text{Cs}(\frac{T}{4})=0,\vspace{0.2cm}\\ \text{Cs}(\theta)>0 \text{ for } \theta\in(-\frac{T}{4},\frac{T}{4}), \text{ and  }\text{Cs}(\theta)<0 \text{ for }\theta\in[-\frac{T}{2},-\frac{T}{4})\cup (\frac{T}{4},\frac{T}{2}],\vspace{0.3cm}\\
\text{Sn}(\frac{T}{2})=\text{Sn}(0)=\text{Sn}(\frac{T}{2})=0,\vspace{0.2cm}\\ \text{Sn}(\theta)>0  \text{ for } \theta\in(0,\frac{T}{2}) , \text{ and  }\text{Sn}(\theta)<0 \text{ for } \theta\in(-\frac{T}{2},0).
\end{array}
\end{equation}

\subsection{Proof of Theorem \ref{teo:main}}\label{proof:ta} We shall prove the analyticity of the $\varphi^+.$ The analyticity of $\varphi^-$ will follows analogously.
	
Using the $(R, \theta, p, q)$-generalized polar change of coordinates for $p=1$ and $q=2k^{+}$ and rescaling the time by taking $\tau =\dfrac{t}{R},$ the vector field \eqref{sistemacanonico} restricted to $y\geq0$ is transformed  into 
\begin{equation}\label{sistemapolar}
(\theta', R') = (F^{+}(R,\theta), G^{+}(R, \theta)), \quad \theta \in [0,\frac{T}{2}] \,\text{ and }\, R>0,
\end{equation}
where
\[
\begin{aligned}
F^{+}(R, \theta)=&  a^{+}  \text{Cs}(\theta )^{2 k^{+} }-2 \delta  k^{+}  \text{Sn}(\theta )\\
&+R \text{Cs}(\theta ) \left(\text{Sn}(\theta ) g^{+}\left(R \text{Cs}(\theta ),R^{2 k^{+} } \text{Sn}(\theta )\right)+\text{Cs}(\theta )^{2 k^{+} } f^{+}(R \text{Cs}(\theta ))\right),\\
G^{+}(R,\theta)=& R \text{Cs}(\theta )^{2 k^{+} -1} \left(\delta  \text{Cs}(\theta )^{2 k^{+} }+a  \text{Sn}(\theta )\right)\\
&+R^2 \text{Sn}(\theta ) \left(\text{Sn}(\theta ) g^{+}\left(R \text{Cs}(\theta ),R^{2 k^{+} } \text{Sn}(\theta )\right)+\text{Cs}(\theta )^{2 k^{+} } f^{+}(R \text{Cs}(\theta ))\right).
\end{aligned}
\]
Notice that, for $R=0,$ $F^{+}(0, \theta)= a^{+}\text{Cs}^{2k^{+}}(\theta) - 2\delta k^+\text{Sn}(\theta).$  Thus, since $\delta a^+ <0,$ and taking \eqref{cssn} into account, we conclude that any root of $F^+(0, \theta)=0$ must satisfy $-\frac{T}{2}<\theta<0.$ Consequently, for $R>0$ sufficiently small, $\theta'>0$ for every $\theta\in[0,\frac{T}{2}].$ This means that $\theta$ can be taken as the independent variable in \eqref{sistemapolar}. Indeed, denoting
\[
H^{+}(R, \theta) = \dfrac{G^{+}(R, \theta)}{F^{+}(R, \theta)}
\]
the differential equation \eqref{sistemapolar} writes
\begin{equation}\label{sistemapolar2}
\dfrac{dR}{d\theta} = H^+(R,\theta),\quad \theta \in [0, \frac{T}{2}]\,\text{ and }\, R>0.
\end{equation}
Since $H^+(R,\theta)$ is analytic in a neighborhood of $\{0\}\times[0,\frac{T}{2}],$ the differential \eqref{sistemapolar2} can be analytically extended to $R=0.$ Accordingly, let $r^{+}(\theta,x_0)$ denote the solution of such an extension satisfying $r^{+}(0,x_0)=x_0.$ From the comments above, we get that $r^{+}(\theta,x_0)$ is analytic in a neighborhood of $[0,\frac{T}{2}]\times\{0\}.$  Finally, notice that $\varphi^{+}(x_0) = r^{+}(\frac{T}{2}, x_0)Cs(\frac{T}{2}).$ Therefore, we conclude that $\varphi^{+}(x_0)$ is analytic in a neighborhood of $x_0=0,$ which concludes the proof of Theorem \ref{teo:main} for the analytic case. The  $C^r$ case is analogous.

\section{A General Property of Involutions}\label{sec:invo}

This section is devoted to prove a general property for pair of involutions. First, we recall the useful  {\it Fa\'{a} di Bruno's Formula} for higher derivatives of a composite function (see \cite{J})
\begin{equation}\label{fadibruno}
\dfrac{d^l}{d\alpha^l}g(h(\alpha))=\displaystyle\sum_{m=1}^l g^{(m)}(h(\alpha)) B_{l,m}\big(h'(\alpha),h''(\alpha),\ldots,h^{(l-m+1)}(\alpha)\big),
\end{equation}
where $B_{l,m}$ denotes the partial Bell polynomials as defined in \eqref{bell}. 

Theorem \ref{teo:oddindex} is a direct consequence of the following property of involutions:

\begin{proposition} \label{prop1}
Let $\varphi,\psi:I\rightarrow \R$ be $C^{2\ell+1}$ involutions around $0.$ If $\varphi(0)=\psi(0)$ and $\varphi^{(i)}(0)=\psi^{(i)}(0)$ for $i=1,2\ldots,2\ell,$ then $\varphi^{(2\ell+1)}(0)=\psi^{(2\ell+1)}(0).$
\end{proposition}

\begin{proof}
Since $\varphi\circ\varphi (x)=x$ and $\psi\circ\psi (x)=x,$ then $\varphi'(0)=-1$ and  $\psi'(0)=-1.$

Now, applying Fa\'{a} di Bruno's Formula \eqref{fadibruno} for computing the $n$-th derivative of the composition $\varphi \circ \varphi(x)=x,$ we get
	\begin{equation}\label{seriesderi}
	\sum^n_{i=1}\varphi^{(i)}(0)\text{B}_{n,i}(\varphi'(0), \varphi''(0), \dots, \varphi^{(n-i+1)}(0))=0,\,\, n\geq 2.
	\end{equation}
	
	Denote $$S({\varphi}'(0), \dots, {\varphi}^{(n-1)}(0)):=-\sum^{n-1}_{i=2}\text{B}_{n,i}({\varphi}'(0), \dots, {\varphi}^{(n-i+1)}(0)).$$
	Thus, from \eqref{seriesderi}, we have 
\[
	 \varphi'(0)\text{B}_{n,1}(\varphi'(0), \dots, \varphi^{(n)}(0)) + \varphi^{(n)}(0)\text{B}_{n,n}(-1)=\displaystyle -\sum^{n-1}_{i=2}\text{B}_{n,i}(\varphi'(0), \dots, \varphi^{(n-i+1)}(0)), \]
	 which implies that
\begin{equation}\label{formulaeqteoC}
	((-1)^n-1)\varphi^{(n)}(0)=S({\varphi}'(0), \dots, {\varphi}^{(n-1)}(0)).
\end{equation}
	Analogously, we obtain that
	\begin{equation} \label{formulaeqteoC2}
	((-1)^n-1)\psi^{(n)}(0) = S({\psi}'(0), \dots, {\psi}^{(n-1)}(0)). 
	\end{equation}
	
	Now, assume that ${\varphi}^{(i)}(0) = {\psi}^{(i)}(0) = \alpha_i,$ for $i = 1, 2, \dots, 2 \ell.$
	From \eqref{formulaeqteoC} and \eqref{formulaeqteoC2}, taking $n = 2\ell+1,$ we get that 
\[
 -2{\varphi}^{(2\ell +1)}(0)= S({\varphi}'(0), \dots, {\varphi}^{(2\ell)}(0)) = S({\alpha_1}, \dots, {\alpha_{2\ell}}) = -2{\psi}^{(2\ell +1)}(0),
\]
which concludes the proof.
\end{proof}

\subsection{Proof of Theorem \ref{teo:oddindex}}\label{proof:tb}
	The proof of Theorem \ref{teo:oddindex} follows directly from Proposition \ref{prop1} by taking $\varphi=\varphi^+$ and $\psi=\varphi^-.$

\section{Lyapunov coefficients}\label{sec:LC}
Consider a Filippov vector field given in the canonical form \eqref{sistemacanonico}. The main idea for determining the coefficients  $\alpha^{+}_n$'s (resp. $\alpha^{-}_n$'s) of the series \eqref{eq:series} of the half-return map $\varphi^{+}$ (resp. $\varphi^{-}$) consists of building a function $\mu^{+}$ (resp. $\mu^{-}$) that maps a point $(x_0,0) \in \Sigma$ onto its image under the flow of $Z^{+}$ (resp. $Z^{-}$) to a transversal section $\Sigma^{\perp}_+= \{(x,y) \in U: x=0,\, y>0 \}$ (resp. $\Sigma_-^{\perp}= \{(x,y) \in U: x=0,\, y<0 \}$) as shown in Fig. \ref{mapadobradobra}. From the transversality of $\Sigma^{\perp}_{\pm}$, the  functions $\mu^{\pm}$ can always be obtained via Implicit Function Theorem. Nevertheless, for the canonical form \eqref{sistemacanonico} the time spent by a trajectory traveling from $(x_0,0) \in \Sigma$ onto $\Sigma^{\perp}_{+}$ (resp. $\Sigma^{\perp}_{-}$) is known, namely, $-\delta x_0$ (resp. $\delta x_0)$, thus the functions $\mu^{\pm}$ can be easily written in terms of the flow. 

Using those functions, we have that the coefficients  $\alpha^{+}_n$'s and $\alpha^{-}_n$'s can be computed implicitly by comparing the series of both sides of the identities
\begin{equation}\label{foldmu}
\mu^{+}(\varphi^{+}(x_0)) = \mu^{+}(x_0)\,\, \text{ and }\, \mu^{-}(\varphi^{-}(x_0)) = \mu^{-}(x_0),
\end{equation}
respectively. This avoids the necessity of blowing-up the singularity and working with generalized polar coordinates to obtain the Lyapunov coefficients.

This alternative for analyzing the half-return maps $\f^{\pm}$ implicitly was introduced by Teixeira in \cite{teixeira77} for studying vector fields defined near the boundary of a manifold and, later, applied in the study of discontinuous planar vector fields \cite{teixeira81}.  In the theory of singularities, the maps $\mu^{\pm}$ satisfying \eqref{foldmu} are known as {\it folds} associated with the involutions $\varphi^{\pm}$. It is worth mentioning that for a given fold, there exists a unique involution for which the fold is associated to (see, for instance, \cite[Proposition 2.6]{mancini05}).

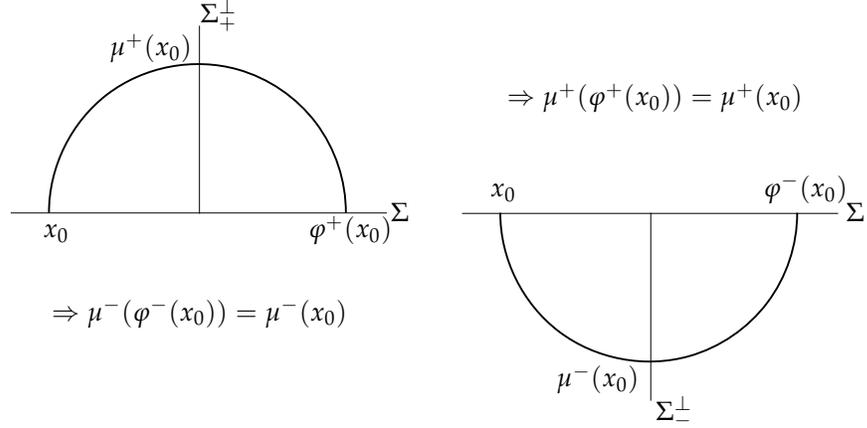
\begin{figure}[H]
	\centering
	\begin{overpic}[scale=0.5]{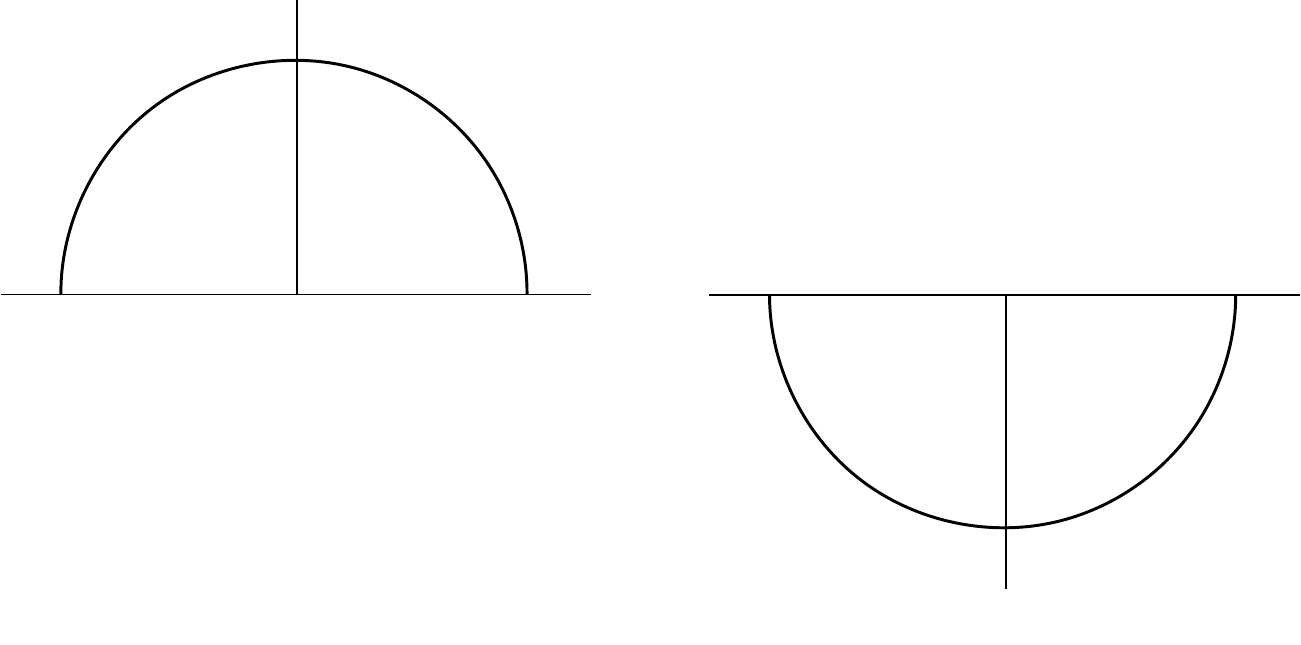}
	\put(5,14){$\Rightarrow \mu^-(\varphi^-(x_0)) = \mu^-(x_0)$}
	\put(60,40){$\Rightarrow \mu^+(\varphi^+(x_0)) = \mu^+(x_0)$}
	\put(46,26){$\Sigma$}
	\put(101,26){$\Sigma$}
	\put(23,50){$\Sigma^{\perp}_+$}
	\put(78,2){$\Sigma^{\perp}_-$}
	\put(12,46){$\mu^+(x_0)$}
	\put(66,6){$\mu^-(x_0)$}
	\put(4,24){$x_0$}
	\put(58,29){$x_0$}
	\put(36,24){$\varphi^+(x_0)$}
	\put(91,29){$\varphi^-(x_0)$}
	\end{overpic}
	\caption{Illustrations of the functions $\mu^{+}$ and $\mu^{-}.$}
	\label{mapadobradobra}
\end{figure}

In what follows, we shall deal with the problem of composition of series \eqref{foldmu}. The reader is referred to \cite{Cima2020} where, although in a different context, similar problems were treated.

\subsection{Preliminary Results}
Denote by $\phi^{\pm}(t,x_0) = (x^{\pm}(t,x_0), y^{\pm}(t,x_0))$ the solutions of 
\[
(\dot{x},\dot{y}) = (\pm \delta, a^{\pm}x^{2k^{\pm}-1} + x^{2k^{\pm}}f^{\pm}(x) + yg^{\pm}(x,y))
\] with initial condition $\phi^{\pm}(0,x_0) = (x^{\pm}(0,x_0), y^{\pm}(0,x_0))=(x_0,0)\in \Sigma.$ Notice that $x^{\pm}(t,x_0)=x_0 \pm \delta t,$ so that $x^{\pm}(t,x_0) = 0$ if, and only if,  $t= \mp \delta x_0.$ Accordingly, we define  
\begin{equation}\label{muproof}
\mu^{\pm}(x_0)=y^{\pm}(\mp \delta x_0,x_0).
\end{equation} 
As commented before, we compute the series of $\mu^{\pm}(x_0)$ around the origin:
\begin{equation*}\label{expansaomapa}
\mu^{\pm}(x_0)= \sum_{n=1}^{\infty}\mu_n^{\pm}x_0^n,
\end{equation*}
 where 
\begin{equation*}
\mu_n^{\pm} = \dfrac{{{\mu^{\pm}}^{(n)}}(0)}{n!}=\sum_{i+j=n}( \mp \delta)^i{n\choose i} \dfrac{\partial^ny^{\pm}(0,0)}{\partial t^i \partial x_0^j} 
=  \dfrac{1}{n!} \sum_{i=1}^{n}(\mp \delta)^i{n \choose i}\dfrac{\partial^{n-i}}{\partial x_0^{n-i}}\left(\dfrac{\partial^i y^{\pm}}{\partial t^i}(0,0) \right) .	
\end{equation*}
In the last equality above, we are using that $y^{\pm}(0,x_0)=0$ and, consequently, $\dfrac{\partial^n y^{\pm}}{\partial x_0^n}(0,0) = 0.$ Now, denoting $y_i^{\pm}(x)=\dfrac{\partial^i y^{\pm}}{\partial t^i}(0,x),$ we get 
\[
\mu_i^{\pm}= \dfrac{1}{i!}\sum_{j=1}^{i} (\mp \delta)^j {{i}\choose{j}} (y_j^\pm)^{(i-j)}(0).
\]
Notice that this last expression coincides with the one presented in \eqref{mui}.

From here, in order to prove Theorem \ref{teo:coef},  we need to establish two preliminary technical lemmas. The first lemma states that \eqref{yi} provides a recursive formula for  $y_i^{\pm}(x),$ and the second lemma establishes that the coefficients $\mu_i^{\pm}$ vanishe for $i\leq2k^{\pm}-1.$  
In their proofs and also in the proof of Theorem \ref{teo:coef}, we shall use some additional identities. Namely:
the well-known {\it General Leibniz Rule} for higher derivatives of product of functions
\begin{equation}\label{prodrule}
\dfrac{d^l}{d\alpha^l}\big(g(\alpha)h(\alpha)\big)= \sum_{k=0}^{l}{l \choose k} g^{(l-k)}(\alpha) h^{(k)}(\alpha);
\end{equation}
and the following {\it Multinomial Formula}  in terms of ordinary Bell polynomials (see \cite{comtet})
\begin{equation}\label{multi}
\Bigg(\sum_{j=1}^{\infty} \alpha_jx^j\Bigg)^n= \sum_{i=n}^{\infty} \hat{B}_{i,n}(\alpha_1, \dots, \alpha_{i-n+1})x^i,
\end{equation}
where $\hat B_{l,m}$ denotes the ordinary Bell polynomials as defined in \eqref{bell}.

\begin{lemma}
\label{lemma:fundamental}
The functions $y_i^{\pm}(x),$ for $i=1,2,\ldots,$ are defined recursively by \eqref{yi}.
\end{lemma}

\begin{proof}
First of all, notice that
\begin{equation*}
\begin{aligned}
\dfrac{\partial y^{\pm}}{\partial t}(t,x)  = & \eta^{\pm}(x \pm \delta t,y^{\pm}(t,x)) \\ 
= & a^{\pm}(x \pm\delta  t)^{2k^{\pm}-1} + (x\pm\delta t)^{2k^{\pm}}f^{\pm}(x\pm\delta t) + y^{\pm}(t,x)g^{\pm} (x \pm\delta t, y^{\pm}(t,x)).
\end{aligned}
\end{equation*} 
Then,
\[
y_1^{\pm} (x)=\dfrac{\partial y^{\pm}}{\partial t}(0,x)=a^{\pm} x^{2k^{\pm}-1}+x^{2k^{\pm}}f^{\pm}(x),
\]
which coincides with the initial condition for $i=1$ of the recursive formula \eqref{yi}.

Now, denoting 
\begin{equation}\label{gi}
g^{\pm}_i(x) =\dfrac{1}{(i-1)!} \dfrac{\partial^{i-1}}{\partial y^{i-1}}g^{\pm}(x,0),
\end{equation}
 we get that
\begin{equation*}
\begin{aligned}
yg^{\pm}(x,y) =  \sum^{\infty}_{m=1} y^m g^{\pm}_m(x).
\end{aligned}
\end{equation*} 
Thus, for $i\geq 2,$
\begin{equation*}
\begin{aligned}
\dfrac{\partial^i y^{\pm}}{\partial t^i}(t,x) 
= &  \dfrac{\partial^{i-1}}{\partial t^{i-1}}\bigg( a^{\pm}(x\pm\delta t)^{2k^{\pm}+1} + (x\pm\delta t)^{2k^{\pm}}f^{\pm}(x \pm\delta  t) +  \sum^{\infty}_{j=1} y^j(t,x) g^{\pm}_j(x \pm\delta  t) \bigg) \\ 
= & a^{\pm} \dfrac{\partial^{i-1}}{\partial t^{i-1}}(x\pm\delta t)^{2k^{\pm}+1} + \dfrac{\partial^{i-1}}{\partial t^{i-1}}\Big((x\pm\delta t)^{2k^{\pm}}f^{\pm}(x \pm\delta  t)\Big)\\
&+  \dfrac{\partial^{i-1}}{\partial t^{i-1}}\bigg(\sum^{\infty}_{j=1} y^j(t,x) g^{\pm}_j(x \pm\delta  t)\bigg). 
\end{aligned}    
\end{equation*}

Clearly,
\begin{equation}\label{deriva1}
\dfrac{\partial^{i-1}}{\partial t^{i-1}}(x\pm\delta t)^{2k^{\pm}+1}\Bigg|_{t=0}= \begin{cases} (\pm \delta)^{i-1} \dfrac{(2k^{\pm}-1)!}{(2k^{\pm}-i)!}x^{2k-i}, \text{ if } i \leq 2k^{\pm}, \\
0, \text{ if } i> 2k^{\pm}.
\end{cases}
\end{equation}
Now, using the Leibniz general rule \eqref{prodrule}, we get that 
\begin{equation}\label{deriva2}
\begin{aligned}
\dfrac{\partial^{i-1}}{\partial t^{i-1}}&\Big((x\pm\delta t)^{2k^{\pm}}f^{\pm}(x \pm\delta  t)\Big)\Bigg|_{t=0}=\\
&\begin{cases}\displaystyle (\pm \delta)^{i-1} \sum_{l=0}^{i-1} {{i-1}\choose{l}}  \dfrac{(2k^{\pm})!}{(2k^{\pm}-l)!}x^{2k^{\pm}-l} {f^{\pm}}^{(i-1-l)} (x),& \text{ if } i \leq 2k^{\pm}, \vspace{0.2cm}\\
\displaystyle (\pm \delta)^{i-1} \sum_{l=0}^{2k^{\pm}-1} {{i-1}\choose{l}}  \dfrac{(2k^{\pm})!}{(2k^{\pm}-l)!}x^{2k^{\pm}-l}{f^{\pm}}^{(i-l-1)}(x),& \text{ if } i> 2k^{\pm}.
\end{cases}
\end{aligned}
\end{equation}
and
\begin{equation}\label{deriva3}
\begin{aligned}
\dfrac{\partial^{i-1}}{\partial t^{i-1}}&\bigg(\sum^{\infty}_{j=1} (y^{\pm}(t,x))^j g^{\pm}_j(x \pm\delta  t)\bigg)\Bigg|_{t=0}=\\
&\sum_{j=1}^{\infty}\sum_{l=0}^{i - 1}{i-1 \choose l}(\pm \delta)^{i-l-1}\dfrac{\partial^l}{\partial t^l}\big(y^{\pm}(t,x)^j\big)\Bigg|_{t=0} {g_j^{\pm}}^{(i-l-1)}(x).
\end{aligned}
\end{equation}
In addition, denoting $P_j(y)=y^j,$ we get from the Fa\'{a} di Bruno's Formula \eqref{fadibruno} that
\begin{equation}\label{deriva3'}
\begin{aligned}
\dfrac{\partial^l}{\partial t^l}(y^{\pm}(t,x)^j)\Bigg|_{t=0}=&\dfrac{\partial^l}{\partial t^l}P_j(y^{\pm}(t,x))\Bigg|_{t=0}= \sum_{m=1}^l  P_j^{(m)}(0) B_{l,m}(y^{\pm}_1(x),\dots , y^{\pm}_{l-m+1}(x))\\
=&\begin{cases}
0, \text{ if } l<j,\vspace{0.2cm}\\
j!B_{l,j}(y^{\pm}_1(x),\dots , y^{\pm}_{l-j+1}(x)), \text{ if } l\geq j.
\end{cases}
\end{aligned}  
\end{equation}
Therefore, substituting \eqref{deriva3'} into \eqref{deriva3} and taking \eqref{gi} into account, we obtain
\begin{equation}\label{deriva4}
\begin{aligned}
\dfrac{\partial^{i-1}}{\partial t^{i-1}}&\bigg(\sum^{\infty}_{j=1} (y^{\pm}(t,x))^j g^{\pm}_j(x \pm\delta  t)\bigg)\Bigg|_{t=0} \\ 
& = \sum_{j=1}^{\infty}\sum_{l=j}^{i-1} {i-1\choose l} (\pm \delta)^{i-l-1}j!B_{l,j}(y^{\pm}_1(x),\dots , y^{\pm}_{l-j+1}(x))  \dfrac{1}{(j-1)!} \dfrac{\partial^{j+i-l-2}g^{\pm}}{\partial x^{i-l-1} \partial y^{j-1}}(x,0)\\
& = \sum_{l=1}^{i-1} \sum_{j=1}^{l}j {i-1\choose l} (\pm \delta)^{i-l-1}B_{l,j}(y^{\pm}_1(x),\dots , y^{\pm}_{l-j+1}(x))  \dfrac{\partial^{j+i-l-2}g^{\pm}}{\partial x^{i-l-1} \partial y^{j-1}}(x,0).
\end{aligned}
\end{equation}

Finally, putting \eqref{deriva1}, \eqref{deriva2}, and \eqref{deriva4} together we get the recursive formula \eqref{yi} for $y_i^{\pm}(x),$ $i\geq2.$
\end{proof}

	\begin{lemma} \label{lema:mu}
	The value $\mu_i$ vanishes for $i=1,\ldots,  2k^{\pm}-1.$
\end{lemma} 
\begin{proof}
	First of all, we proceed by induction in order to prove that 
	\begin{equation} \label{eq:indmu}
	y_i^{\pm}(x) = x^{2k^{\pm} - i}R_i^{\pm}(x)
	\text{ for } i \leq 2k^{\pm}, 
	\end{equation} where $R_i^{\pm}(x)$ is a smooth function. 
	For $i=1,$ \eqref{eq:indmu} holds. Indeed,
	\[ 
	y_1^{\pm} (x)= a^{\pm} x^{2k^{\pm}-1}+x^{2k^{\pm}}f^{\pm}(x) =x^{2k^{\pm} - 1}(a^{\pm} + xf^{\pm}(x)).\] 
	Now, let $i \leq 2k^{\pm}.$ Recall that, from \eqref{yi}, 
	\begin{equation*}
	\begin{aligned}
	y_i^{\pm} (x) = &  a^{\pm}(\pm \delta)^{i-1} \dfrac{(2k^{\pm}-1)!}{(2k^{\pm}-i)!}x^{2k-i}  +\sum_{l=0}^{i-1} {{i-1}\choose{l}}  \dfrac{(2k^{\pm})!}{(2k^{\pm}-l)!}x^{2k^{\pm}-l} {f^{\pm}}^{(i-1-l)} (x) & \\  & +\sum_{l=1}^{i-1} \sum_{j=1}^{l}j {i-1\choose l} (\pm \delta)^{i-l-1}B_{l,j}(y^{\pm}_1(x),\dots , y^{\pm}_{l-j+1}(x))  \dfrac{\partial^{j+i-l-2}g^{\pm}}{\partial x^{i-l-1} \partial y^{j-1}}(x,0).
	\end{aligned}
	\end{equation*}
	Suppose that \eqref{eq:indmu} holds for all $s \leq i-1,$ that is, $y_s^{\pm}(x) = x^{2k^{\pm} - s}R_s(x).$ Then, taking into account that $B_{l,j}$ is a homogeneous polynomial of degree $j$ with $l-j+1$ variables, we have that 
	\[\sum_{l=1}^{i-1}\sum_{j=1}^{l}j {s-1\choose l} (\pm \delta)^{s-l-1}B_{l,j}(y^{\pm}_1(x),\dots , y^{\pm}_{l-j+1}(x))  \dfrac{\partial^{j+i-l-2}g^{\pm}}{\partial x^{i-l-1} \partial y^{j-1}}(x,0) = x^{2k^{\pm}-i+1}T(x),
	\] where $T$ is a smooth function. Then, 
	\begin{equation*}
	\begin{aligned}
	y_i^{\pm} (x) = &  a^{\pm}(\pm \delta)^{i-1} \dfrac{(2k^{\pm}-1)!}{(2k^{\pm}-i)!}x^{2k^{\pm}-i}   +\sum_{l=0}^{i-1} {{i-1}\choose{l}}  \dfrac{(2k^{\pm})!}{(2k^{\pm}-l)!}x^{2k^{\pm}-l} {f^{\pm}}^{(i-1-l)} (x) + x^{2k^{\pm}-i+1}T(x) \\
	= & x^{2k^{\pm}-i}R_i^{\pm}(x),
	\end{aligned}
	\end{equation*} which implies that  \eqref{eq:indmu} holds for all $i \leq 2k^{\pm}.$ 
	
	From \eqref{mui} and \eqref{eq:indmu}, we conclude that
	\begin{equation*}
	\mu_i^{\pm}= \dfrac{1}{i!}\sum_{j=1}^{i} (\pm \delta)^j {{i}\choose{j}} \dfrac{\partial^{i-j}y_j^{\pm}}{\partial x^{i-1}}(0) = \dfrac{1}{i!}\sum_{j=1}^{i} (\pm \delta)^j {{i}\choose{j}}  \dfrac{\partial^{i-j}}{\partial x^{i-1}}\left(x^{2k^{\pm}-j}R_j(x)\right) \bigg|_{x=0}=0,
	\end{equation*}
		for $i\leq 2k^{\pm} - 1.$
\end{proof}

\subsection{Proof of Theorem \ref{teo:coef}}\label{proof:coef}
Suppose that $Z^{\pm}$ are analytic vector fields and assume that the piecewise analytic Filippov vector field \eqref{sistemainicial} has a $(2k^+,2k^-)$-monodromic tangential singularity at the origin, for positive integers $k^+$ and $k^-.$ From the comments of Section \ref{sec:cf}, we know that there exists a small neighborhood $U\subset\R^2$ of the origin such that \eqref{sistemainicial} is equivalent to the canonical form \eqref{sistemacanonico} through a time rescaling, where $\delta,$ $a^{\pm},$  $f^{\pm}(x),$ and $g^{\pm}(x,y)$ are given by \eqref{delta}, \eqref{value:a}, and \eqref{auxfunc},  respectively. In addition, since  $X^{\pm}(x,y)\neq0$ for every $(x,y)\in U,$  we get that the functions $f^{\pm}$ and $g^{\pm}$ are analytic in a neighborhood of $x=y=0.$ 

Now, working out the identity $\mu^{\pm}(x_0)=\mu^{\pm}(\varphi^{\pm}(x_0)) ,$ we obtain
	\begin{equation} \label{muinicial}
	\begin{aligned}
	\sum_{n=1}^{\infty}\mu^{\pm}_nx_0^n =& \sum_{n=1}^{\infty}\mu^{\pm}_n({ \sum_{j=1}^{\infty}\alpha^{\pm}_jx_0^j})^n \\
	=& \sum_{n=1}^{\infty}\mu_n^{\pm}\sum_{i=n}^{\infty}\hat{B}_{i,n}(\alpha_1^{\pm}, \dots, \alpha^{\pm}_{i-n+1})x_0^i \\
	=&  \sum_{n=1}^{\infty}\sum_{i=n}^{\infty}\mu_n^{\pm}\hat{B}_{i,n}(\alpha_1^{\pm}, \dots, \alpha^{\pm}_{i-n+1})x_0^i. \\
	\end{aligned}
	\end{equation} 
	In the second equality above, we are using the multinomial formula \eqref{multi}.

	By comparing the coefficients of $x_0^i$ in both sides of equality \eqref{muinicial} and taking Lemma \ref{lema:mu} into account, we conclude that
	\begin{equation*} \label{coef}
	\mu_i^{\pm}= \sum_{j=2k^{\pm}}^{i} \mu_j^{\pm} \hat{B}_{i,j}(\alpha_1^{\pm}, \dots, \alpha^{\pm}_{i-j+1}),\quad i \geq 2k^{\pm}.
	\end{equation*}
	
	Doing some computations, we have that 
\begin{equation} \label{expressaomun}
	\begin{aligned}
	\mu_i^{\pm} =& \mu_{2k^{\pm}}^{\pm}\hat{B}_{i,2k^{\pm}}(\alpha_1^{\pm}, \dots, \alpha^{\pm}_{i-2k^{\pm}+1}) + \sum_{j=2k^{\pm} + 1}^{i} \mu_j^{\pm} \hat{B}_{i,j}(\alpha_1^{\pm}, \dots, \alpha^{\pm}_{i-j+1}) \\
	=& \mu_{2k^{\pm}}^{\pm}\hat{B}_{i, 2k^{\pm}}(\alpha_1^{\pm}, \dots, \alpha^{\pm}_{i-2k^{\pm}},0) + \mu_{2k^{\pm}}^{\pm}{\alpha_1^{\pm}}^{2k^{\pm}-1}\alpha^{\pm}_{i-2k^{\pm}+1}2k^{\pm} \\ &+\sum_{j=2k^{\pm} + 1}^{i} \mu_j^{\pm} \hat{B}_{i,n}(\alpha_1^{\pm}, \dots, \alpha^{\pm}_{i-j+1}).
	\end{aligned}
\end{equation}
	
	Therefore, isolating $\alpha^{\pm}_{i-2k^{\pm}+1}$ in \eqref{expressaomun}, we get
	\begin{equation}\label{alphapre}
	\alpha^{\pm}_{i-2k^{\pm}+1} = \dfrac{\mu_i^{\pm} -\mu_{2k^{\pm}} \hat B_{i,2k^{\pm}}\big(\alpha_1^{\pm},\ldots,\alpha^{\pm}_{i-2k^{\pm}},0\big) - \displaystyle \sum_{j=2k^{\pm}+1}^{i} \mu_j \hat B_{i,j}\big(\alpha_1^{\pm},\ldots,\alpha^{\pm}_{i-j+1}\big)}{{\alpha^{\pm}_1}^{2k^{\pm}-1}2k^{\pm}\mu_{2k^{\pm}}}.
	\end{equation}
	Finally, taking into account that $\alpha_1^{\pm}=-1 $ and doing a change of the index in \eqref{alphapre}, we obtain the recurrence \eqref{alpha} for $\alpha_i^{\pm}.$ This concludes the proof of Theorem \ref{teo:coef}.

\section{Bifurcation of limit cycles}\label{sec:limc}

This section is devoted to the proofs of Theorems \ref{thm:hopf} and \ref{thm:genhopf}, which are based on the {\it Malgrange Preparation Theorem} (see \cite{malgrange1971}) and {\it Implicit Function Theorem}.

\subsection{Proof of Theorem \ref{thm:hopf}}

Let $Z_\lambda$ be a 1-parameter family of Filippov vector fields \eqref{sistemainicial} having a $(2k^+,2k^-)$-monodromic tangential singularity at the origin for every $\lambda$ in an interval $I.$ Let $V_2(\lambda),$ $V_3(\lambda),$ and $V_4(\lambda)$ be, respectively, the second, the third, and the forth Lyapunov coefficients. 
Accordingly, the displacement function of $Z_{\lambda}$ around the origin writes 
\begin{equation}\label{DeltaEx}
\Delta(x; \lambda)  = V_2(\lambda) x^2 + V_3(\lambda) x^3 +V_4(\lambda) x^4 + O(x^5)=x^2\Gamma(x;\lambda)
\end{equation}
where
\begin{equation} \label{thmhopf:dispfun}
\begin{aligned}
\Gamma(x; \lambda) & = V_2(\lambda) + V_3(\lambda) x +V_4(\lambda) x^2 + O(x^3).
\end{aligned}
\end{equation}

By hypothesis, there exists $\lambda_0\in I$ such that $V_2(\lambda_0)=0,$ $V_2'(\lambda_0)=d\neq0,$ and $V_4(\lambda_0)=\ell\neq0.$ Thus,
\[
\Gamma(0;\lambda_0 )= V_2(\lambda_0) = 0, \quad \dfrac{\partial^2 \Gamma}{\partial x^2}(0;\lambda_0) = 2V_4(\lambda_0) \neq 0, 
\]
and, from Theorem \ref{teo:oddindex},
\[
\dfrac{\partial \Gamma}{\partial x}(0;\lambda_0) = V_3(\lambda_0) = 0.
\]
Therefore, as a consequence of the Malgrange Preparation Theorem (see \cite{malgrange1971}), there exists a small neighborhood $W\subset \R^2$ of $(0,\la_0)$ and smooth functions $c(x;\la),$ $a_0(\la),$ and $a_1(\la)$ such that $c(x,\la)\neq0$ and
\begin{equation}\label{thmhopf:dispfun2}
	\Gamma(x;\lambda) = c(x, \lambda) (x^2 + a_1(\lambda)x + a_0(\lambda )) ,
\end{equation} for every $(x,\lambda)\in W.$

From \eqref{thmhopf:dispfun} and \eqref{thmhopf:dispfun2} we obtain that 
\[
a_0(\lambda) =  \dfrac{V_2(\lambda)}{c(0,\lambda)}\,\text{ and }\,
a_1(\lambda) =\dfrac{V_3(\lambda) - \partial_x c(0,\lambda) a_0(\lambda)}{c(0,\lambda) }.
\]
Consequently, $a_0(\la_0)=a_1(\la_0)=0.$ In addition, one can see that 
\[
c(0,\la_0)=V_4(\la_0)=\ell\,\text{ and }\, a_0'(\la_0)=\dfrac{V_2'(\la_0)}{c(0,\la_0)}=\dfrac{d}{\ell}.
\]

Now,  taking the hypothesis $d\ell(\la-\la_0)<0$ into account, we can easily compute the unique positive root of \eqref{thmhopf:dispfun2} in $W$ as
\[
\begin{aligned}
x^*=& \dfrac{-a_1(\lambda) + \sqrt{a_1(\lambda)^2 - 4a_0(\lambda)}}{2}= \sqrt{\dfrac{-V_2'(\lambda_0)  (\lambda - \lambda_0)}{V_4(\lambda_0)}} +\mathcal{O}(\lambda-\lambda_0)\\
 = & \sqrt{\dfrac{-d (\lambda - \lambda_0)}{\ell}} +\mathcal{O}(\lambda-\lambda_0).
\end{aligned}
\] 
Thus, there exists a unique limit cycle bifurcating from the origin which intersects the discontinuity manifold for $x>0$ at $(x^*(\la),0),$ which lies $\sqrt{|\lambda - \lambda_0|}$-close to the origin. Moreover, the stability of such a limit cycle coincides with the stability of the monodromic singularity at the origin for $\la=\la_0,$ that is, it is asymptotically stable (resp. unstable) provided that $\ell<0$ (resp. $\ell>0$). This information could also be obtained by computing the derivative of \eqref{DeltaEx} at $x=x^*.$

\subsection{Proof of Theorem \ref{thm:genhopf}}

Let $Z_\Lambda$ be an $n$-parameter family of Filippov vector fields \eqref{sistemainicial} having a $(2k^+,2k^-)$-monodromic tangential singularity at the origin for every $\Lambda$ in an open set $U\subset\R^n.$ Let $V_{i}(\Lambda)$ be the $i-$th Lyapunov coefficient, for $i=1,2\dots,2n+2.$ Accordingly, the displacement function of $Z_{\lambda}$ around the origin writes
\begin{equation*}
	\Delta(x; \Lambda) = \sum_{i=2}^{2n+2} V_i(\Lambda) x^i+ O(x^{2n+3})=x^2 \Gamma(x; \Lambda),
\end{equation*}
where
\begin{equation} \label{eq1:genhopf}
 	\Gamma(x;\Lambda) =\sum_{i=2}^{2n+2} V_i(\Lambda) x^{i-2} + O(x^{2n+1}).
\end{equation} 
Notice that
\begin{equation*}
\begin{aligned}
\dfrac{\partial^i \Gamma}{\partial x^i}(0;\Lambda) = i! V_{i+2}(\Lambda), \,\text{ for } \, i=0,\ldots, 2n. 
\end{aligned}
\end{equation*} 

By hypothesis, there exists $\Lambda_0\in U$ such that $\mathcal V_n(\Lambda_0)=0,$ $\det (D\mathcal V_n(\Lambda_0))\neq0,$ and $V_{2n+2}(\Lambda_0)\neq0,$ where $\mathcal V_n=(V_2,V_4,\ldots,V_{2n}):U\rightarrow \R^n.$ Thus, 
\[
\dfrac{\partial^{2i} \Gamma}{\partial x^{2i}}(0;\Lambda_0) = 0, \,\text{ for }\, i =1,\dots, n-1, \,\text{ and }\,
	\dfrac{\partial^{2n} \Gamma}{\partial x^{2n}}(0;\Lambda_0) = (2n)! V_{2n+2}(\Lambda_0)  \neq 0.
\]
In addition, from Theorem \ref{teo:oddindex},
\[
\dfrac{\partial^{2i+1} \Gamma}{\partial x^{2i+1}}(0;\Lambda_0) = 0, \,\text{ for }\, i =1,\dots, n-1.
\]
Therefore, as a consequence of the Malgrange Preparation Theorem (see \cite{malgrange1971}), there exists a small neighborhood $W\subset \R\times\R^n$ of $(0,\Lambda_0)$ and smooth functions $c(x;\Lambda),$ and $a_i(\Lambda),$ for $i=0,\ldots,2n-1,$ such that $c(x,\Lambda)\neq0$ and
\begin{equation} \label{eq2:genhopf}
\begin{aligned}
	\Gamma(x; \Lambda) = & c(x,\Lambda)(x^{2n} + a_{2n-1}(\Lambda)x^{2n-1} + \dots + a_1(\Lambda)x+a_0(\Lambda)),
\end{aligned}
\end{equation}
for every $(x,\Lambda)\in W.$

From  \eqref{eq1:genhopf} and \eqref{eq2:genhopf}, we have that 	
\[
i! V_{i+2}(\Lambda)=\dfrac{\partial^i \Gamma}{\partial x^i}(0; \Lambda) =\sum_{j=0}^{i} {i \choose j} \partial^{i-j} c(0, \Lambda) j! a_j(\Lambda).
\]
Hence, denoting $\mathcal{A}(\Lambda)=\Big(a_0(\Lambda),\ldots,a_{2n-1}(\Lambda)\Big)$ and $\mathcal{V}(\Lambda)=\Big(2! V_2(\Lambda),\ldots,(2n-1)! V_2(\Lambda)\Big),$ we see that
\[
M(\Lambda) \mathcal{A}(\Lambda)=\mathcal{V}(\Lambda),
\]
where $M(\Lambda)$ is a lower triangular matrix with every entry in the diagonal given by $c(0,\Lambda).$ Therefore, $M(\Lambda)$ is invertible for every $\Lambda$ in a small neighborhood of $\Lambda_0$ and
\[
 \mathcal{A}(\Lambda)=M(\Lambda)^{-1}\mathcal{V}(\Lambda).
\]
Consequently,   $\mathcal{A}(\Lambda_0) = (0,\ldots,0)$  
and
$D\mathcal{A}(\Lambda_0) = M^{-1}(\Lambda_0)D\mathcal{V}(\Lambda_0).$
Thus, since by hypothesis $D\mathcal{V}_n(\Lambda_0)$ is invertible, we conclude that $D\mathcal{V}(\Lambda_0)$ and, consequently, $D\mathcal{A}(\Lambda_0)$ are full rank matrices, that is, $\textrm{rank}(D\mathcal{V}(\Lambda_0))=\textrm{rank}(D\mathcal{A}(\Lambda_0)) = n.$
	
Now, for $(x,\Lambda)\in W,$ denote 
\[
P(x,\Lambda)  =\dfrac{\Gamma(x;\Lambda)}{c(x, \Lambda) }= x^{2n} + a_{2n-1}(\Lambda) x^{2n-1}+ \dots +  a_1(\Lambda)x + a_0(\Lambda).
\]
Let $x_i,$ $i=1,2,\ldots,n,$ be distinct $n$  positive values and $\e>0.$ In what follows, we will conclude the proof of Theorem \ref{thm:genhopf} by showing that for $\e>0$ sufficiently small there exists $\Lambda^*(\e)$ sufficiently close to $\Lambda_0$ such that
\[
 P(\e x_i, \Lambda^*(\e)) = 0, \quad \text{for}\quad i=1, 2, \dots, n.
 \]
 This will imply that $Z_{\Lambda^*(\e)}$ has $n$ limit cycles bifurcating from the origin for $\e>0$ sufficiently small.
 
First, consider the system of equations
\begin{equation}\label{systeoE}
 P(\e x_i, \Lambda) = 0, \quad \text{for}\quad i=1, 2, \dots, n,
\end{equation}
which is equivalent to
\begin{equation}\label{sys2teoE}
N(\e) \mathcal{A}(\Lambda)=b(\e),
\end{equation}
where 
\[
N(\e)=\left(\begin{array}{cccc} 
1 & \e x_1 & \dots & (\e x_1)^{2n-1} \\
1 & \e x_2 & \dots & (\e x_2)^{2n-1} \\
\vdots & \vdots & \ddots & \vdots  \\
1 & \e x_n & \dots & (\e x_n)^{2n-1}
\end{array} \right) \quad \text{and}\quad b(\e)=-\left(\begin{array}{c} 
	(\e x_1)^{2n} \\
	(\e x_2)^{2n} \\
	\vdots \\
	(\e x_n)^{2n}
 \end{array} \right).
\]
Notice that the matrix $N(\e)$ is composed by two blocks, $N(\e)=\Big(\,T(\e)\quad S(\e)\,\Big),$ where $T(\e)$ and $S(\e)$ are square matrices given by
\[
T(\e)=\left(\begin{array}{cccc} 
1 & \e x_1 & \dots & (\e x_1)^{n-1} \\
1 & \e x_2 & \dots & (\e x_2)^{n-1} \\
\vdots & \vdots & \ddots & \vdots  \\
1 & \e x_n & \dots & (\e x_n)^{n-1}
\end{array} \right) \quad \text{and}\quad S(\e)=\left(\begin{array}{ccc} 
 (\e x_1)^{n}  & \dots & (\e x_1)^{2n-1} \\
 (\e x_1)^{n}  & \dots & (\e x_2)^{2n-1} \\
\vdots  & \ddots & \vdots  \\
 (\e x_1)^{n} &  \dots & (\e x_n)^{2n-1}
\end{array} \right).
\]
Since $T(\e)$ is a {\it Vandermonde Matrix}, we know that
\[ 
\det (T(\e))= \e^n\prod_{1\leq i<j\leq n}(x_j-x_i)
\] 
and, therefore, invertible for $\e\neq0.$
Thus, consider the smooth matrix-valued functions  $\widetilde N(\e)=T(\e) ^{-1} N(\e)$ and $\tilde b(\e)=T(\e) ^{-1} b(\e)$ which, because of the factor $\e^n$ in both $S(\e)$ and $b(\e),$ can be smoothly extended for $\e=0$ as $\widetilde N(0)=\Big(\,I_n\quad 0_n\,\Big)$  and $\tilde b(0)=0.$ Now, define the function $F:\R^{n}\times \R\rightarrow \R^{n}$ as
\[
F(\Lambda,\e)=\widetilde N(\e) \mathcal{A}(\Lambda)- \tilde b(\Lambda).
\]
Clearly, the systems of equations \eqref{systeoE} and \eqref{sys2teoE} are equivalent to $F(\Lambda,\e)=0.$  Notice that $F(\Lambda_0,0)=0$ and, since $\widetilde N(0)$ and  $D\mathcal{A}(\Lambda_0)$  are full rank matrices, we conclude that the square matrix $\dfrac{\partial F}{\partial \Lambda}(\Lambda_0,0)=\widetilde N(0) D\mathcal{A}(\Lambda_0)$  has full rank and, therefore, is non-singular. Then, from the implicit function theorem, we obtain for $\e>0$ sufficiently small a smooth function $\Lambda^*(\e)$ such that $\Lambda(0)=\Lambda_0$ and $F(\Lambda^*(\e),\e)=0$ for $\e>0$ sufficiently small.  This concludes the proof of Theorem \ref{thm:genhopf}.

\section*{Appendix A: Algorithms}\label{sec:alg}

\lstset{language=Mathematica}
\lstset{basicstyle={\sffamily\footnotesize},
	numbers=left,
	numberstyle=\tiny\color{gray},
	numbersep=5pt,
	breaklines=true,
	captionpos={t},
	frame={lines},
	rulecolor=\color{black},
	framerule=0.5pt,
	columns=flexible,
	tabsize=2
}

In this appendix, based on Theorem \ref{teo:coef}, we present an implemented Mathematica algorithm for computing the coefficients $\alpha^{+}_n$ and $\alpha^{-}_n$ of the series \eqref{eq:series} of the half-return maps $\varphi^{+}$ and $\varphi^{-}$ and, consequently, the Lyapunov coefficients $V_n$'s.

In what follows, we are denoting  $ k^+ = $ {\sffamily kp},  $k^-=$ {\sffamily kn}, $a^+ =$ {\sffamily ap} , $ a^-=$ {\sffamily an},  $\mu^+=\mu${\sffamily p}, and  $\mu^-=\mu${\sffamily n}. In addition, {\sffamily yp0[i]} and {\sffamily yp1[i]} denote $y_i^+,$ respectively, for $i \leq 2k^+$ and for $i > 2k^+.$ Analogously, {\sffamily yn0[i]} and {\sffamily yn1[i]} denote $y_i^-,$ respectively, for $i \leq 2k^-$ and for $i > 2k^-.$

In order to run the codes for computing the first $N$ Lyapunov coefficients, we have to specify the values for {\sffamily kp, kn}, $\delta,$ and {\sffamily imax}=$N.$

\begin{lstlisting}[language=Mathematica,caption={Mathematica's algorithm for computing $y_i^{\pm}.$ \vspace{0.05cm}},mathescape=true]
yp0[1] = ap x^(2 kp - 1) + x^(2 kp) fp[x] 
yp0[i_] := $\delta$^(i -1) (ap (2 kp - 1)!/(2 kp - i)! x^(2 kp - i) + Sum[Binomial[i - 1, l] (2 kp)!/(2 kp - l)! x^(2 kp - l) D[fp[x], {x, i - 1 - l}], {l, 0, i - 1}]) + Sum[Sum[j Binomial[i - 1, l] $\delta$^(i - l - 1) BellY[l, j,Yp[l - j + 1, x]] (D[D[gp[x, y], {y, j - 1}], {x, i - 1 - l}] /. y -> 0), {j, 1, l}], {l, 1, i - 1}] 
yp1[i_] := $\delta$^(i - 1) (Binomial[i - 1, 2 kp] (2 kp)! D[fp[x], {x, i - 1 - 2 kp}] + Sum[Binomial[i - 1, l] (2 kp)!/(2 kp - l)! x^(2 kp - l) D[fp[x], {x, i - 1 - l}], {l, 0, 2 kp - 1}]) + Sum[Sum[j Binomial[i - 1, l] $\delta$^(i - l - 1) BellY[l, j, Yp[l - j + 1,x]] (D[D[gp[x, y], {y, j - 1}], {x, i -1 - l}] /. y -> 0), {j,1, l}], {l, 1, i - 1}]
Yp[1] = {yp0[1]};
For[i = 2, i <= 2 kp, i++, Yp[i] = Join[Yp[i - 1],{yp0[i]}];]
For[i = 2 kp + 1, i <= 2 kp + imax, i++, Yp[i] = Join[Yp[i - 1], {yp1[i]}];]
For[i = 1, i <= 2 kp + imax, i++, yp[i] = Yp[2 kp + imax][[i]]]

yn0[1] = an x^(2 kn - 1) + x^(2 kn) fn[x]
yn0[i_] := (-$\delta$)^(i - 1) (an (2 kn - 1)!/(2 kn - i)! x^(2 kn - i) + Sum[Binomial[i - 1, l] (2 kn)!/(2 kn - l)! x^(2 kn - l) D[fn[x], {x, i - 1 - l}], {l, 0, i - 1}]) + Sum[Sum[j Binomial[i - 1, l] (-$\delta$)^(i - 1 - l) BellY[l, j, Yn[l - j + 1,x]] (D[D[gn[x, y], {y, j - 1}], {x, i - 1 - l}] /. y -> 0), {j,1, l}], {l, 1, i - 1}]
yn1[i_] := (-$\delta$)^(i -1) (Binomial[i - 1, 2 kn] (2 kn)! D[fn[x], {x, i - 1 - 2 kn}] + Sum[Binomial[i - 1, l] (2 kn)!/(2 kn - l)! x^(2 kn - l) D[fn[x], {x, i - 1 - l}], {l, 0, 2 kn - 1}]) + Sum[Sum[j Binomial[i - 1, l] (-$\delta$)^(i - 1 - l) BellY[l, j,Yn[l - j + 1, x]] (D[D[gn[x, y], {y, j - 1}], {x, i - 1 - l}] /. y -> 0), {j, 1, l}], {l, 1, i - 1}]
Yn[1] = {yn0[1, x]};
For[i = 2, i <= 2 kn, i++, Yn[i] = Join[Yn[i - 1], {yn0[i]}];]
For[i = 2 kn + 1, i <= 2 kn + imax, i++, Yn[i] = Join[Yn[i - 1], {yn1[i]}];]
For[i = 1, i <= 2 kn+ imax, i++, yn[i] = Yn[2 kn + imax][[i]]]
\end{lstlisting}

\begin{lstlisting}[language=Mathematica,caption={Mathematica's algorithm for computing  $\mu_n^{\pm}.$ \vspace{0,05cm}},mathescape=true]
$\mu$p[n_] := 1/n! Sum[(-$\delta$)^j Binomial[n, j] D[yp[j, x], {x, n - j}] /. x -> 0, {j, 1, n}]
$\mu$n[n_] := 1/n! Sum[$\delta$^j Binomial[n, j] D[yn[j, x], {x, n - j}] /.x -> 0, {j, 1, n}]
\end{lstlisting}

\begin{lstlisting}[language=Mathematica,caption={Mathematica's algorithm for computing  $\alpha_n^{\pm}$ and $V_n$  \vspace{0,05cm}},mathescape=true]
$\alpha$p[1] = -1;
Ap[1] = {$\alpha$p[1]};
For[n = 2, n <= imax, n++, $\alpha$p[n] = Factor[($\mu$p[2 kp] (2 kp)!/(n + 2 kp - 1)! BellY[n + 2 kp - 1, 2 kp, Join[Ap[n - 1], {0}]] + Sum[$\mu$p[i] i!/(n + 2 kp - 1)! BellY[n + 2 kp - 1, i, Ap[n + 2 kp - i]], {i, 2 kp + 1, n + 2 kp - 1}] - $\mu$p[ n + 2 kp - 1])/(2 kp $\mu$p[2 kp])]; Ap[n] = Join[Ap[n - 1], {n! $\alpha$p[n]}];]

$\alpha$n[1] = -1;
An[1] = {$\alpha$n[1]};
For[n = 2, n <= imax, n++, $\alpha$n[n] = Factor[($\mu$n[2 kn] (2 kn)!/(n + 2 kn - 1)! BellY[n + 2 kn - 1, 2 kn, Join[An[n - 1], {0}]] + Sum[$\mu$n[i] i!/(n + 2 kn - 1)! BellY[n + 2 kn - 1, i, An[n + 2 kn - i]], {i, 2 kn + 1, n + 2 kn - 1}] - $\mu$n[n + 2 kn - 1])/(2 kn $\mu$n[2 kn])]; An[n] = Join[An[n - 1], {n! $\alpha$n[n]}];]

For[n = 1, n <= imax, j++, V[n] = $\delta(\alpha$p[n] - $\alpha$n[n]);]
\end{lstlisting}

\section*{Appendix B: Proof of Corollary \ref{cor}}\label{sec:cor}

In this appendix, we present a brief proof of Corollary \ref{cor}. Recall that,  for $k^{\pm}$ known, the algorithms presented in Appendix A already give us a simple way for computing all the Lyapunov coefficients.
Thus, we start this proof by computing the coefficients $\alpha_n^{\pm}$, $n=1,2,3,4,$ for $k^{\pm}=1$. 
\begin{equation}\label{eq:coro1k1}
\begin{aligned}
\alpha_1^{\pm} = &  -1, \quad \alpha_2^{\pm} =  \dfrac{-2f_0 \pm \delta 2g_{0,0}}{3a}, \quad \alpha_3^{\pm} =  - {\alpha_2^{\pm}}^2, \\
\alpha_4^{\pm} = & \dfrac{1}{135a^3}\Big(80f_0^3 +(\mp \delta)150af_0^2g_{0,0}+132a^2f_0g_{0,0}^2 + (\mp \delta) 44 a^3 g_{0,0}^3 -90 a f_0 f_1 \\
& \pm \delta 36 a^2g_{0,0} f_1 +54 a^2f_2 +36a^4g_{0,1} \pm \delta 18 a^2f_0g_{1,0}-18a^3g_{0,0}g_{1,0}+(\mp \delta)18 a^3g_{2,0} \Big). \\
\end{aligned}
\end{equation}

From now on, we shall consider $k \geq 2$. 
In order to compute the remaining coefficients, the following partial Bell polynomials are needed:		\begin{equation}\label{eq:bellpoly}
		\begin{aligned}
		B_{n,n}(x_1) =&  (x_1)^n, \\
		B_{n,n-1}(x_1, x_2) =&  \binom{n}{2}(x_1)^{n-2}x_2, \\
		B_{n,n-2}(x_1, x_2,x_3) =  &{n \choose 3}(x_1)^{n-3}x_3 + 3{n \choose 4}(x_1)^{n-4}(x_2)^2, \\
		B_{n,n-3}(x_1, x_2, x_3, x_4) =& {n \choose 4}(x_1)^{n-4}x_4 + 10{n \choose 5}(x_1)^{n-5}x_2x_3 + 15{n \choose 6}(x_1)^{n-6}x_2^3. \\
		\end{aligned}
		\end{equation}
		
		Now, developing the recursive formula given by \eqref{alpha}  for $\alpha_i$ and using \eqref{eq:bellpoly}, we get that
		\begin{equation}\label{eq:coro1}
		\begin{aligned}
		\alpha_1^{\pm} = &  -1, \quad \alpha_2^{\pm} =   \dfrac{-\mu^{\pm}_{2k +1}}{k^{\pm} \mu^{\pm}_{2k}}, \quad \alpha_3^{\pm} =   - {\alpha_2^{\pm}}^2,   \\
		\alpha_4^{\pm} = & \dfrac{(7+k(3+2k)) {\mu^{\pm}_{2k+1}}^{3} - 6k(1+k) \mu^{\pm}_2k \mu^{\pm}_{2k+1} \mu^{\pm}_{2k+2} + 6k^2{\mu^{\pm}_{2k}}^2 \mu^{\pm}_{2k+3} }{6k^3 {\mu^{\pm}_{2k}}^3}. \\
		\end{aligned}
		\end{equation}
		From \eqref{eq:coro1},  we have to compute $\mu^{\pm}_{2k}, \mu^{\pm}_{2k+1}, \mu^{\pm}_{2k+2},$ and $\mu^{\pm}_{2k+3}$. Recall that we are denoting 
		\[
		f^{\pm}(x)=\sum_{i=0}^{\infty} f_i^{\pm}\, x^i\quad \text{and} \quad g^{\pm}(x,y)=\sum_{i=0}^{\infty} \sum_{j=0}^{\infty}  g_{i,j}^{\pm}\, x^i y^j.
		\]  Throughout the proof, we shall also drop the sign $\pm$  from $k^{\pm}$, $a^{\pm},$ $f_i^{\pm},$ and $g_{i,j}^{\pm}$.
		
		\bigskip		
		
\noindent{\bf Computation of $\mu^{\pm}_{2k}$.} 
		From \eqref{mui}, we have that 
		\begin{equation}\label{mu2kf1}
		\mu_{2k}^{\pm} = \dfrac{1}{(2k)!} \sum^{2k}_{j=1}(\mp \delta)^j {2k \choose j} (y_j^{\pm})^{(2k-j)}(0).
		\end{equation} 
		Taking \eqref{yi} into account, it follows that
		\begin{equation}\label{mu2kf2}
		(y_j^{\pm})^{(2k-j)}(0) = (\pm \delta)^{j-1}{a}(2k-1)!.
		\end{equation} Then, substituting \eqref{mu2kf2} in \eqref{mu2kf1}, we obtain
		\begin{equation}\label{mu2kf3}
		\mu_{2k}^{\pm}  =  \dfrac{1}{(2k)!} \sum^{2k}_{j=1}(\mp \delta)^j {2k \choose j} (\mp \delta)^{j-1}{a}(2k-1)!  = \dfrac{(\mp \delta) {a}}{2k}. 
		\end{equation}
	
	\bigskip	
	
\noindent{\bf Computation of $\mu^{\pm}_{2k+1}$.}
		From \eqref{mui}, we have that 
		\begin{equation}\label{mu2k1f1}
		\mu_{2k+1}^{\pm}  = \dfrac{1}{(2k+1)!} \sum^{2k+1}_{j=1}(\mp \delta)^j {2k+1 \choose j} (y_j^{\pm})^{(2k+1-j)}(0).
		\end{equation}
Taking \eqref{yi} into account, it follows that
		\begin{equation}\label{mu2k1f2}
		(y_j^{\pm})^{(2k +1 -j)}(0)  = \begin{cases} 
		(2k)!f_0,\quad \text{ if } j=1,\vspace{0.1cm} \\
		(\pm \delta)^{j-1} (2k)! f_0 + (\pm\delta)^j{a}(2k-1)!g_{0,0},\quad \text{ if }2 \leq j \leq 2k+1.
		\end{cases}
		\end{equation}
		Although formula \eqref{yi} distinguishes the cases $j<2k+1$ and $j=2k+1,$  when developing this formula we see that these cases can be put together as \eqref{mu2k1f2}. 
		Now, substituting \eqref{mu2k1f2} into \eqref{mu2k1f1}, we obtain 
		\begin{equation}\label{mu2k1f3}
		\begin{aligned}
		\mu_{2k+1}^{\pm}  =&   \dfrac{1}{(2k+1)!} \Big((\mp \delta){2k+1 \choose 1}(2k)!f_0  +  \sum^{2k+1}_{j=2} (\mp \delta)^j {2k + 1 \choose j} \big( (\pm \delta)^{j-1} (2k)! f_0  \\ & + (\pm\delta)^j{a}(2k-1)!g_{0,0}\big)\Big)=  \dfrac{(\mp \delta)f_0+{a}g_{0,0}}{2k+1}.  \\
		\end{aligned}
		\end{equation}
		
	\bigskip	
\noindent{\bf Computation of $\mu^{\pm}_{2k+2}$.}
		From \eqref{mui}, we have that

		\begin{equation}\label{mu2k2f1}
		\mu_{2k+2}^{\pm}  = \dfrac{1}{(2k+2)!} \sum^{2k+2}_{j=1}(\mp \delta)^j {2k+2 \choose j} (y_j^{\pm})^{(2k+2-j)}(0).
		\end{equation}
Taking \eqref{yi} into account, it follows that
		\begin{equation}\label{mu2k2f2}
		\begin{aligned}
		(y_1^{\pm})^{(2k +1)}(0) =& 
		(2k+1)!f_1,  \\
		(y_2^{\pm})^{(2k)}(0)=& (\pm \delta)\Big(\dfrac{(2k)!}{(2k-1)!} 2f_1+  2f_1 \Big) + 2\Big(f_0g_{0,0} + {a}g_{1,0} \Big), \\
		(y_j^{\pm})^{(2k+2-j)}(0) =& \dfrac{2k+2-j}{2!} \bigg( (\pm \delta)^{j-1} \Big(\dfrac{(2k)!}{(2k+1-j)!}2 f_1  \\ 
		&+ {j-1 \choose j-2}\dfrac{(2k)!}{(2k+2-i)!} 2 f_1 \Big) + 2\Big( \big((\pm \delta)^{j-2} \dfrac{(2k)!}{(2k+2-j)!}f_0 \\ 
		& + {a}(\pm \delta)^{j-3} \dfrac{(2k-1)!}{(2k+2-j)!}g_{0,0} \big) g_{0,0} \\
		& + {a}(\pm \delta)^{j-2} \dfrac{(2k-1)!}{(2k+1-j)!}g_{1,0} \Big) \\ 
		& +\big(2 {a} {j-1 \choose j-2} (\pm \delta)^{j-4} \dfrac{(2k-1)!}{(2k+2-j)!}  g_{1,0}\big)  \bigg),  \text{ if } 3 \leq j \leq 2k, \\
		(y_{2k+1}^{\pm})'(0)=& (2k)!f_1+ \dfrac{(2k)!}{(2k-1)!}(2k)!f_1 + \big( (\pm \delta)(2k)!f_0  \\ 
		&+ {a}(2k-1)!g_{0,0} \big) g_{0,0} \\ 
		&+ (\pm \delta){a}(2k-1)!g_{1,0}  + {2k \choose 2k-1}(\pm \delta) a^{\pm }(2k-1)!g_{1,0},  \\
		y_{2k+2}^{\pm}(0)=& {2k+1 \choose 2k} (2k)!(\pm \delta)f_1 + \big( (2k)!f_1 + (\pm \delta){a}(2k-1)!g_{0,0} \big)g_{0,0} \\
		& + {2k+1 \choose 2k}{a}(2k-1)!g_{1,0}. \\ 
		\end{aligned}
		\end{equation}
		Substituting \eqref{mu2k2f2} in \eqref{mu2k2f1} and proceeding with algebraic manipulations, we obtain 
		\begin{equation}\label{mu2k2f3}
		\begin{aligned}
		\mu_{2k+2}^{\pm} = & \dfrac{2f_0g_{0,0}+ (\mp \delta) {a} g_{0,0}^2 + 2 (\mp \delta)f_1+ {a}g_{1,0}}{4k+4}. \\ 
		\end{aligned}
		\end{equation}

	\bigskip	
\noindent{\bf Computation of $\mu^{\pm}_{2k+3}$.}
		From \eqref{mui}, we have that

		\begin{equation}\label{mu2k3f1}
		\mu_{2k+3}^{\pm}  = \dfrac{1}{(2k+3)!} \sum^{2k+3}_{j=1}(\mp \delta)^j {2k+3 \choose j} (y_j^{\pm})^{(2k+3-j)}(0). 
		\end{equation}
		Taking \eqref{yi} into account, it follows that
		\begin{align*}
		(y_1^{\pm})^{(2k +2)}(0)  = & 
		(2k+2)!f_2,  \\
		(y_2^{\pm})^{(2k+1)}(0) = & (\pm \delta) (2k+1)!\big( 2k f_2 + 12f_2\big)+ (y_1^{\pm})^{(2k +1)}(0)g_{0,0}   \\
		& + 2(y_1^{\pm})^{(2k)}(0)g_{1,0} + 2(y_1^{\pm})^{(2k-1)}(0)g_{2,0}, \\
		(y_3^{\pm})^{(2k)}(0) = & \dfrac{(2k)!}{3!} \big( \dfrac{(2k)!}{(2k-2!)}6f_2+ 3\dfrac{(2k)!}{(2k-1)!} 12f_2  + 12f_2\big)+ (y_{2}^{\pm})^{(2k)}(0)g_{0,0}   \\
		& + 2(y_{2}^{\pm})^{(2k-1)}(0)g_{1,0} + 2(y_{2}^{\pm})^{(2k-2)}(0)g_{2,0} \\
		& + 3(\pm \delta)\big((y_{1}^{\pm})^{(2k)}(0)g_{1,0} + 2(y_{1}^{\pm})^{(2k-1)}(0)g_{2,0} \big),  \\
		(y_j^{\pm})^{(2k+3-j)}(0)  = & \dfrac{(2k+3-j)!}{3!} \bigg( (\pm \delta)^{j-1} \Big(\dfrac{(2k)!}{(2k+1-j)!}6 f_2  + {j-1 \choose j-2}\dfrac{(2k)!}{(2k+2-i)!} 12 f_2  \\ 
		& +  {j-1 \choose j-3} \dfrac{(2k-1)!}{(2k+3-j)!} 12f_2\Big)\bigg) + (y_{j-1}^{\pm})^{(2k+3-j)}(0)g_{0,0}   \\
		& + 2(y_{j-1}^{\pm})^{(2k+2-j)}(0)g_{1,0} + 2(y_{j-1}^{\pm})^{(2k+1-j)}(0)g_{2,0} \\
		& + {j-1 \choose j-2}(\pm \delta)\big((y_{j-2}^{\pm})^{(2k+3-j)}(0)g_{1,0} + 2(y_{j-2}^{\pm})^{(2k+2-j)}(0)g_{2,0} \big)\\
		& +  2{j-1 \choose j-3}(y_{j-3}^{\pm})^{(2k+3-j)}(0)g_{2,0}, \text{ if } 4 \leq j \leq 2k,  \\
		(y_{2k+1}^{\pm})''(0) = & (2k)!2f_2+ {2k \choose 2k-1}(2k)!4f_2+{2k \choose 2k-2}\dfrac{(2k)!}{2!}4f_2  \\
		& + {(y_{2k}^{\pm})}''(0)g_{0,0}+ 2{(y_{2k}^{\pm})}'(0)g_{1,0} + 2y_{2k}^{\pm}(0)g_{2,0} \\
		& + 2{2k \choose 2k-1}(\pm \delta)\Big(\big(\dfrac{(2k)!}{2!}f_0+{a}(\pm \delta) \dfrac{(2k-1)!}{2!}g_{0,0}\big)g_{1,0} \\
		& + 2{a}(2k-1)!g_{2,0}\Big)+ 4 {2k \choose 2k-2} (\pm \delta) {a}\dfrac{(2k-1)!}{2!}g_{2,0}, \\
		(y_{2k+2}^{\pm})'(0) = & {2k+1 \choose 2k} (2k)!(\pm \delta)2f_2 + (\pm \delta){2k+1 \choose 2k-1} (2k)!2f_2 \\
		& + \Big( (2k)!f_1 + {2k \choose 2k-1}(2k)!f_1 + \big((\pm \delta)(2k)!f_0 + {a}(2k-1)!g_{0,0}\big)g_{0,0} \\
		& + (\pm \delta){a}(2k-1)!g_{1,0}  + {2k \choose 2k-1}(2k-1)!g_{1,0} \Big)g_{0,0} \\
		& + \big( (2k)!f_0+ (\pm \delta){a}(2k-1)!g_{0,0}\big)g_{1,0} \\
		& + {2k+1 \choose 2k}(\pm \delta)\Big(\big((\pm \delta)(2k)!f_0+{a}(2k-1)!g_{0,0}\big)g_{1,0} \\
		& +2(\pm \delta){a}(2k-1)!g_{2,0} \Big) + 2{2k+1 \choose 2k-1} {a}(2k-1)!g_{2,0},\\
		(y_{2k+3}^{\pm})'(0) = & {2k+2 \choose 2k} (2k)!2f_2 + \bigg({2k+1 \choose 2k}(2k)!(\pm \delta)f_1 \\
		&+ \big((2k)!f_0 + {a}(\pm \delta)(2k-1)!g_{0,0}\big)g_{0,0}  + {2k+1 \choose 2k}{a}(2k-1)!g_{1,0} \bigg)g_{0,0} \\
		& + {2k+2 \choose 2k+1}(\pm \delta)\big( (2k)!f_0+ (\pm \delta){a}(2k-1)!g_{0,0}\big)g_{1,0} \\
		& + 2{2k+2 \choose 2k}(\pm \delta){a}(2k-1)!g_{2,0}.\\
		\end{align*}
		Notice that, at this point, the computations start to become more cumbersome. Also, in the formulae above, there are some values of $y_i^{(j)}(0)$ which are not explicitly computed. However, except for $(y_{j-1}^{\pm})^{(2k+1-j)}(0) = {a}(\pm \delta)^{j-2} \dfrac{(2k-1)!}{(2k+1-j)!}$, the others can be computed by using the formulae \eqref{mu2kf2}, \eqref{mu2k1f2}, and \eqref{mu2k2f2}. Then, substituting all these values into \eqref{mu2k3f1} and proceeding with algebraic manipulations, we obtain
		\begin{equation}\label{mu2k3f3}
		\begin{array}{ll}
		\mu_{2k+3}^{\pm} = & \dfrac{1}{6(3+3k)}  \Big(6 g_{0,0} f_1+6(\mp \delta)f_2+3 f_0 \big((\mp \delta)g_{0,0}^2 +g_{1,0}\big)\\
		&+{a}\big( g_{0,0}^3+3(\mp \delta) g_{0,0}g_{1,0} + 2g_{2,0}\big)\Big). \\ 
		\end{array}
		\end{equation}
		
		\bigskip
		
		Finally, substituting \eqref{mu2kf3}, \eqref{mu2k1f3}, \eqref{mu2k2f3}, and \eqref{mu2k3f3} into \eqref{eq:coro1} and proceeding with algebraic manipulations, we conclude that
		\begin{equation}\label{a4km2}
	\begin{aligned}
	\alpha_1^{\pm}=&  -1,\quad  \alpha_2^{\pm}=  \dfrac{-2f_0 \pm 2 \delta a g_{0,0}}{2a k + a},\quad \alpha_3^{\pm}= -(\alpha_2^{\pm})^2, \vspace{0.2cm}\\
	\alpha_4^{\pm}=& \dfrac{4({k} (2 {k}+3)+7) (-f_0 \pm \delta {a}   {g_{0,0}})^3}{3{a}^3(2 {k}+1)^3} \mp  \dfrac{ 12 \delta a(f_0\mp \delta {a}   {g_{0,0}}) \left({a} \left({g_{1,0}} \mp \delta  {g_{0,0}}^2\right)+2 f_0 {g_{0,0}} \mp 2 \delta  {f_1}\right)}{3{a}^3(8{k}+4)} \vspace{0.2cm}\\  &\pm  \dfrac{ 4\delta {a}^2 \left({a} \left({2g_{2,0}}+{g_{0,0}}^3\right)+6 {g_{0,0}} {f_1}+3 f_0 {g_{1,0}}\mp 3 \delta  \left({a} {g_{1,0}} {g_{0,0}}+f_0 {g_{0,0}}^2+{2f_2}\right) \right)}{3 {a}^3 (8 {k}+12)},
	\end{aligned}
		\end{equation} for $k \geq 2$. This proof follows by comparing expressions \eqref{eq:coro1k1} and \eqref{a4km2} with the ones provided in the statement of the of Corollary \ref{cor}.
\section*{Acknowledgments}

The authors thank the referee for the constructive comments and
suggestions which led to an improved version of the manuscript.

DDN is partially supported by FAPESP grants 2018/16430-8, 2019/10269-3, and 2018/13481-0, and by CNPq grants 306649/2018-7 and 438975/2018-9. LAS is partially supported by CAPES 001.

\bibliographystyle{abbrv}
\bibliography{references}

\end{document}